\RequirePackage{ifthen}
\newboolean{SIOPT}
\setboolean{SIOPT}{false}

\ifthenelse {\boolean{SIOPT}}
{
\documentclass[review,hidelinks,onefignum,onetabnum]{siamart220329}
\usepackage[margin=1.5in]{geometry}

\newtheorem{statement}{Statement}
\newtheorem{example}{Example}

} {
\documentclass[11pt,letterpaper]{article}
\usepackage[margin=1in]{geometry}

\usepackage{amsthm}
\newtheorem{theorem}{Theorem}
\newtheorem{proposition}{Proposition}
\newtheorem{lemma}{Lemma}
\newtheorem{corollary}{Corollary}

\newtheorem{example}{Example}

\newtheorem{statement}{Statement}

}

\usepackage{hyperref}
\usepackage{amsmath}

\ifthenelse {\boolean{SIOPT}}
{
\crefname{problem}{Problem}{Problems}
\crefname{statement}{Statement}{Statements}

\newcommand{\parag}[1]{\paragraph{#1}}

} {
\usepackage[capitalize,noabbrev]{cleveref}
\crefname{question}{Question}{Questions}
\crefname{step}{Step}{Steps}
\crefname{claim}{Claim}{Claims}
\crefname{problem}{Problem}{Problems}
\crefname{observation}{Observation}{Observations}
\crefname{statement}{Statement}{Statements}

\newcommand{\parag}[1]{\paragraph{#1.}}

}

\usepackage{graphicx} % Required for inserting images

\usepackage{amssymb}
\usepackage{color}

 \usepackage[usenames,dvipsnames]{pstricks}
 \usepackage{epsfig}
 \usepackage{pst-grad} % For gradients
 \usepackage{pst-plot} % For axes
 \usepackage[space]{grffile} % For spaces in paths
 \usepackage{etoolbox} % For spaces in paths
 \makeatletter % For spaces in paths
 \patchcmd\Gread@eps{\@inputcheck#1 }{\@inputcheck"#1"\relax}{}{}
 \makeatother

\usepackage{tikz}
\usetikzlibrary{shapes.geometric, arrows}
\usetikzlibrary{positioning}

\tikzstyle{startstop} = [rectangle, rounded corners, 
minimum size=6mm, 
align=center, 
draw=black]

\usepackage{enumitem} % to use [noitemsep]

\newcommand{\MP}{\text{MP}}

\def\P{\mathcal{P}}
\def\NP{\mathcal{NP}}
\def\BPP{\mathcal{BPP}}

\newcommand{\ie}{i.e., }
\renewcommand{\S}{\mathcal S}

\DeclareMathOperator{\gap}{gap}
\DeclareMathOperator{\mh}{mh}
\DeclareMathOperator{\tw}{tw}
\DeclareMathOperator{\BQP}{BQP}
\DeclareMathOperator{\xc}{xc}

\def\R{{\mathbb R}}

\newcommand{\bra}[1]{\left\{#1\right\}}

\allowdisplaybreaks
\newcommand{\N}{\mathcal N}

\DeclareMathOperator{\poly}{poly}

\newcommand{\PBS}{\text{PBS}}
\newcommand{\PBP}{\text{PBP}}

\title{Beyond hypergraph acyclicity: limits of tractability for pseudo-Boolean optimization}
\author{Alberto Del Pia
\thanks{Department of Industrial and Systems Engineering \& Wisconsin Institute for Discovery,
             University of Wisconsin-Madison.
             E-mail: {\tt delpia@wisc.edu}.
             }
\and
Aida Khajavirad
\thanks{Department of Industrial and Systems Engineering,
             Lehigh University.
             E-mail: {\tt aida@lehigh.edu}.
             }
}

\date{January 31, 2025}

\begin{document}

\maketitle

\begin{abstract}
In this paper, we study the problem of minimizing a polynomial function with literals over all binary points, often referred to as \emph{pseudo-Boolean optimization.}
%a popular encoding of binary polynomial optimization.
We investigate the fundamental limits of computation for this problem by providing new necessary conditions and sufficient conditions for tractability.
On the one hand, we obtain the first intractability results, in the best-case sense, for pseudo-Boolean optimization problems on signed hypergraphs with bounded rank, in terms of the treewidth of the intersection graph. 
Namely, first, under some mild assumptions, we show that for every sequence of hypergraphs indexed by the treewidth and with bounded rank, the complexity of solving the associated pseudo-Boolean optimization problem grows super-polynomially in the treewidth. Second, we show that any hypergraph of bounded rank is the underlying hypergraph of some signed hypergraph for which the corresponding pseudo-Boolean polytope has an exponential extension complexity in the treewidth.
On the other hand, we introduce the nest-set gap, a new hypergraph-theoretic notion that enables us to define a notion of ``distance'' from the hypergaph acyclicity. We prove that if this distance is bounded, the pseudo-Boolean polytope admits a polynomial-size extended formulation. This in turn enables us to obtain a polynomial-time algorithm for a large class of pseudo-Boolean optimization problems whose underlying hypergraphs contain $\beta$-cycles.
\end{abstract}

\newcommand{\kywrds}{binary polynomial optimization; pseudo-Boolean optimization; treewidth; pseudo-Boolean polytope; extended formulations; extension complexity}

\ifthenelse {\boolean{SIOPT}}
{
% For SIOPT begin

% For SIOPT end
}{
% For OO begin
}

%\tableofcontents

\section{Introduction}
%\paragProblem statement} 
Binary polynomial optimization, \ie the problem of maximizing a multivariate polynomial function over the set of binary points, is a fundamental NP-hard problem in discrete optimization. In this paper, we consider a formulation that encodes the objective function using pseudo-Boolean functions, often referred to as pseudo-Boolean optimization. To formally define the problem, we make use of signed hypergraphs, a representation scheme that was recently introduced in~\cite{dPKha23}. Recall that a \emph{hypergraph} $G$ is a pair $(V,E)$, where $V$ is a finite set of nodes and $E$ is a set of subsets of $V$ of cardinality at least two, called the edges of $G$. 
A \emph{signed hypergraph} $H$ is a pair $(V,S)$, where $V$ is a finite set of nodes and $S$ is a set of signed edges.
A \emph{signed edge} $s \in S$ is a pair $(e_s,\eta_s)$, where $e_s$ is a subset of $V$ of cardinality at least two, and $\eta_s$ is a map that assigns to each $v \in e_s$ a \emph{sign} $\eta_s(v) \in \{-1,+1\}$.
The \emph{underlying edge} of a signed edge $s=(e_s,\eta_s)$ is $e_s$.
%Two signed edges $s=(e_s,\eta_s)$, $s'=(e_{s'},\eta_{s'}) \in S$ are said to be \emph{parallel} if $e_s=e_{s'}$.
%, and they are said to be \emph{identical} if $e_s=e_{s'}$ and $\eta_s=\eta_{s'}$. 
%Throughout this paper, we consider signed hypergraphs with no identical signed edges. 
%However, our signed hypergraphs often contain parallel signed edges. 
It is important to remark that the signed hypergraphs we consider in this paper often contain several signed edges with a common underlying edge but different signs.
With any signed hypergraph $H = (V,S)$, and cost vector $c \in \R^{V \cup S}$, we associate the \emph{pseudo-Boolean optimization problem:}
\begin{align}
\label[problem]{prob PBO}
\tag{PBO}
\begin{split}
\max & \qquad \sum_{v\in V} {c_v z_v} + \sum_{s \in S} c_s \prod_{v \in e_s} \sigma_s(z_v) \\
{\rm s.t.} & \qquad z \in \{0,1\}^V,
\end{split}
\end{align}
where 
$$
\sigma_s(z_v) := 
\begin{cases}
z_v & \text{ if } \eta_s(v) = +1 \\ 
1-z_v & \text{ if } \eta_s(v) = -1,
\end{cases}
$$
and without loss of generality we assume $c_s \neq 0$ for all $s \in S$.
Several applications such as maximum satisfiability problems~\cite{goemans94}, and graphical models~\cite{jordan04} can naturally be formulated as pseudo-Boolean optimization problems. 
See~\cite{BorHam02} for a review of existing results on pseudo-Boolean optimization.
We then linearize the objective function of~\cref{prob PBO} by introducing a variable $z_s$, for each signed edge $s \in S$, to obtain an equivalent formulation in a lifted space:
\begin{align}
\label[problem]{prob LPBO}
\tag{L-PBO}
\begin{split}
\max & \qquad \sum_{v\in V} {c_v z_v} + \sum_{s \in S} {c_s z_s} \\
{\rm s.t.} & \qquad z_s = \prod_{v \in e_s} {\sigma_s(z_v)} \qquad \forall s \in S \\
& \qquad z \in \{0,1\}^{V \cup S}.
\end{split}
\end{align}
% To solve \cref{prob LPBO} efficiently using polyhedral techniques, it is essential to understand the facial structure of the polyhedral convex hull of its feasible region. 
% To this end, 
%in the same vein as~\cite{dPKha17MOR}, 
In~\cite{dPKha23}, the authors introduced the~\emph{pseudo-Boolean set} 
of the signed hypergraph $H=(V,S)$, as the feasible region of \cref{prob LPBO}:
\begin{equation*} 
%\label{eq PBP}
\PBS(H) := 
\Big\{ z \in \{0,1\}^{V \cup S} : z_s = \prod_{v \in e_s} {\sigma_s(z_v)}, \; \forall s \in S \Big\},
\end{equation*}
and they refer to its convex hull as the \emph{pseudo-Boolean polytope} and denote it by $\PBP(H)$. 
With each signed hypergraph $H=(V,S)$, we associate two key hypergraphs: 
\begin{itemize}
\item [(i)] The~\emph{underlying hypergraph} of $H$, which is the hypergraph obtained from $H$ by ignoring the signs. 
Formally, it is the hypergraph $(V,E)$, where $E := \bra{e_s : s = (e_s,\eta_s) \in S}$.
\item [(ii)] The \emph{multilinear hypergraph} of $H$, which is the hypergraph $\mh(H) = (V,E)$, where $E$ is constructed as follows: For each $s \in S$, and every $t \subseteq e_s$ with $\eta_s(v)=-1$ for all $v \in t$, the set $E$ contains $\{v \in e_s : \eta_s(v)=+1\} \cup t$, if this set has cardinality at least two.
\end{itemize}

Let us consider an important special case of~\cref{prob PBO} obtained by letting $\eta_s(v) = +1$ for every $s \in S$ and $v \in e_s$. In this case, a signed hypergraph $H = (V,S)$ essentially coincides with its underlying hypergraph $G=(V,E)$, and hence~\cref{prob PBO} can be equivalently written as the following \emph{binary multilinear optimization problem}:
\begin{align}
\label[problem]{prob BPO}
\tag{BMO}
\begin{split}
\max & \qquad \sum_{v\in V} {c_v z_v} + \sum_{e\in E} {c_e \prod_{v\in e} {z_v}} \\
{\rm s.t.} & \qquad z \in \{0,1\}^V,
\end{split}
\end{align}
where again without loss of generality we assume $c_e \neq 0$ for all $e \in E$.  
In~\cite{dPKha17MOR}, the authors defined the \emph{multilinear set} as the feasible region of a linearized multilinear optimization problem: 
\begin{equation*}
\S(G):= \Big\{ z \in \{0,1\}^{V \cup E} : z_e = \prod_{v \in e} {z_{v}}, \; \forall e \in E \Big\},
\end{equation*}
and referred to its convex hull as the~\emph{multilinear polytope} $\MP(G)$. Notice that by expanding the objective function of~\cref{prob PBO} over a signed hypergraph $H$, this problem can be reformulated in the form of~\cref{prob BPO} over the multilinear hypergraph of $H$; \ie $\mh(H)$. 
In the special case where $|e|=2$ for all $e \in E$,~\cref{prob BPO} simplifies to a~\emph{binary quadratic optimization (BQO)} problem and the multilinear polytope coincides with the well-known Boolean quadric polytope introduced by Padberg~\cite{Pad89} and later studied by others.

\medskip

In this paper, we investigate the fundamental limits of computation for pseudo-Boolean optimization. Via innovative applications of the inflation operation introduced in~\cite{dPKha23}, we provide the first necessary conditions in the best-case sense, as well as new sufficient conditions for the tractability of this class of optimization problems.

%\subsection{Polynomial-time solvable classes}
%\subsection{Computational complexity of binary polynomial optimization}

\subsection{Extended formulations and polynomial-time solvable classes}

In the following, we review the literature on polynomial-size extended formulations for the multilinear polytope and the pseudo-Boolean polytope. 
If a polynomial-size extended formulation of $\PBP(H)$ (resp. $\MP(G)$) can be constructed in polynomial time, then~\cref{prob PBO} (resp.~\cref{prob BPO}) can be solved in polynomial time by solving a linear program. 
However, in general, the converse may not hold. The \emph{extension complexity} of a polytope is defined as the minimum number of inequalities and equalities in an extended formulation of the polytope. Rothvoss~\cite{rothvoss17} showed that the matching polytope has an exponential extension complexity, whereas the maximum weight matching problem can be solved in polynomial time. Yet, extension complexity is a quantity of great interest in discrete optimization, as, for example, a super-polynomial extension complexity implies that the optimization problem cannot be formulated as a polynomial-size linear program.

\parag{Acyclic hypergraphs}
In~\cite{Pad89}, Padberg proved that if $G=(V,E)$ is an acyclic graph, then the Boolean quadric polytope $\BQP(G)$  
is defined by $4|E|$ inequalities in the original space.
The notion of graph acyclicity has been extended to several different notions of hypergraph acyclicity; in increasing order of generality, one can name Berge-acyclicity, $\gamma$-acyclicity, $\beta$-acyclicity, and $\alpha$-acyclicity.
We should remark that polynomial-time algorithms for determining acyclicity level of hypergraphs are available~\cite{fag83}.
In~\cite{dPKha18SIOPT,BucCraRod18,dPDiG23ALG,dPKha23mMPA}, the authors obtained a complete characterization of acyclic hypergraphs whose multilinear polytopes admit polynomial-size extended formulations.  
Henceforth, we denote by $r$ the \emph{rank} of the hypergraph $G=(V,E)$, defined as the maximum cardinality of an edge in $E$.
In~\cite{dPKha18SIOPT,BucCraRod18}, the authors proved that if $G$ is Berge-acyclic, then
$\MP(G)$ is defined by $|V|+(r+2)|E|$ inequalities in the original space.
Moreover, in~\cite{dPKha18SIOPT}, the authors proved that if $G$ is
$\gamma$-acyclic, then $\MP(G)$ has
a polynomial-size extended formulation with
at most $|V|+2|E|$ variables and at most $|V|+(r+2)|E|$ inequalities.
%the authors present a strongly polynomial-time algorithm to solve the separation problem.
%Subsequently, in~\cite{dPKha21MOR}, we  introduce~\emph{running intersection inequalities}, a class of facet-defining inequalities that serve as a generalization of flower inequalities. We prove that
%for kite-free $\beta$-acyclic hypergraphs, a class that lies
%between $\gamma$-acyclic and $\beta$-acyclic hypergraphs, the polytope obtained by adding all running intersection inequalities to the standard linearization coincides with $\MP(G)$, and it admits a polynomial-size extended formulation with at most $|V|+2|E|$ variables and at most $|V|+(r+2)|E|$ inequalities.
In~\cite{dPKha23mMPA}, the authors present a polynomial-size extended formulation for the multilinear polytope of $\beta$-acyclic hypergraphs  with at most $(r-1)|V|+|E|$ variables and at most $(3r-4)|V|+4|E|$ inequalities. In~\cite{dPKha23}, the authors introduced the pseudo-Boolean polytope and subsequently generalized the earlier results by proving that if the underlying hypergraph of a signed hypergraph $H=(V,S)$ is $\beta$-acyclic, then $\PBP(H)$ has a polynomial-size extended formulation with $O(r|V||S|)$ variables and inequalities. Note that this result is more general than the previous ones because it only requires the $\beta$-acyclicity of the underlying hypergraph of $H$. Indeed, the multilinear hypergraph of $H$ may contain many $\beta$-cycles. 
On the other hand, in~\cite{dPDiG23ALG}, the authors prove that~\cref{prob BPO} is strongly NP-hard over $\alpha$-acyclic hypergraphs.
This result implies that, unless $\P = \NP$, one cannot construct a polynomial-size extended formulation for the multilinear polytope of $\alpha$-acyclic hypergraphs.
However, in~\cite{dPKha21MOR}, the authors showed that if the rank $r$ of an $\alpha$-acyclic hypergraph is upper bounded by the logarithm of a polynomial in the size of the hypergraph, then the multilinear polytope has a polynomial-size extended formulation with $O(2^r|V|)$ variables and inequalities.

\parag{Bounded treewidth} Another celebrated line of research relates the complexity of combinatorial optimization problems to the treewidth of a corresponding graph. In~\cite{WaiJor04,Lau09,BieMun18} the authors proved that if the intersection graph of the hypergraph has a bounded treewidth, then the multilinear polytope has a polynomial-size extended formulation. 
Recall that given a hypergraph $G=(V,E)$, the \emph{intersection graph of $G$} is the graph with node set $V$, and where two nodes
$v,v' \in V$ are adjacent if $v,v' \in e$ for some $e \in E$.
Interestingly, this sufficient condition is equivalent to the bounded rank assumption for $\alpha$-acyclic hypergraphs (see section~4 of~\cite{dPKha21MOR} for the proof of equivalence). 
In~\cite{CapdPDiG24mIPCO}, the authors show that if the incidence graph of the underlying hypergraph has a bounded treewidth, then the pseudo-Boolean polytope has a polynomial-size extended formulation. 
Recall that given a hypergraph $G=(V,E)$, the~\emph{incidence graph of $G$}
is a bipartite graph whose vertex set is $V \cup E$ and the edge set is $\{\{v, e\} : v \in V, \ e \in E, \ v \in e\}$.
For any hypergraph, the treewidth of the incidence graph is upper bounded by the treewidth of the intersection graph. While the problem of computing the treewidth of a graph is NP-hard in general, it is fixed-parameter tractable when parameterized by the treewidth~\cite{BodKos08}. 

\medskip

It is important to note that in all of the above results regarding acyclic hypergraphs and bounded treewidth, the proposed extended formulations can be constructed in polynomial-time, thereby translating into polynomial-time algorithms to solve the corresponding~\cref{prob PBO} or~\cref{prob BPO}.  
For further results on polyhedral relaxations of multilinear sets of degree at least three, 
see~\cite{CraRod17,dPKhaSah20MPC, dPKha21MOR,dPDiG21IJO,kha22,CheDasGun23,dPWal23MPB,KimRicTaw24}.

\parag{The extension complexity} In~\cite{fiorini12}, the authors prove that the Boolean quadric polytope of a complete graph has exponential extension complexity. In~\cite{AboFio19}, the authors prove that for any proper minor-closed family of graphs, there exists
a positive constant $\delta$ such that the Boolean quadric polytope of any graph $G$ in the family has extension complexity $2^{\delta\tw(G)}$, where $\tw(G)$ denotes the treewidth of $G$.
From the proof of their main result it also follows that there exists a positive constant $\delta > \frac{1}{20}$, such that the Boolean quadric polytope of any graph $G$ has an extension complexity lowered bounded by  
$2^{\Omega((\tw(G))^\delta)}$. These results imply that, for graphs with unbounded treewidth, the Boolean quadric polytope does not admit a polynomial-size extended formulation. To this date, no result regarding the extension complexity of the multilinear polytope or the pseudo-Boolean polytope is available.

%In this paper, we obtain the first lower bounds on the extension complexity of the pseudo-Boolean polytope on signed hypergraphs of bounded rank.

\subsection{Our contributions}

%Moreover, we say that a hypergraph or a signed hypergraph has a \emph{log-poly} attribute (e.g. rank or treewidth), if that attribute can be upper bounded by the logarithm of a polynomial in the size of of the hypergraph.
In this paper, we obtain new necessary conditions and sufficient conditions for the existence of polynomial-size extended formulations for the pseudo-Boolean polytope, as well as new necessary conditions for polynomial-time solvability of pseudo-Boolean optimization problems. The common technique to establish all our results is the innovative use of the
\emph{inflation operation} recently introduced in~\cite{dPKha23}.
Henceforth, for brevity, we define the \emph{treewidth} of a hypergraph as the treewidth of the intersection graph of the hypergraph; the \emph{treewidth} of a signed hypergraph is then defined as the treewidth of underlying hypergraph of the signed hypergraph.

\medskip

The main contributions of this paper are twofold:

    \parag{Negative results}
    In Section~\ref{sec:limits}, we obtain the first intractability results, in the best-case sense, for pseudo-Boolean optimization. 
    %Namely, under some mild assumptions, we show that there exists no sequence of hypergraphs $\bra{G_k}_{k=1}^\infty$ indexed by the treewidth $k$, with bounded rank and with unbounded treewidth, for which every instance of~\cref{prob PBO} on a signed hypergraph whose underlying hypergraph is $G_k$ can be solved in polynomial time in $k$ (see~\cref{{th PBO hard}}). 
    First, under some mild assumptions, we show that for every sequence of hypergraphs $\bra{G_k}_{k=1}^\infty$ indexed by treewidth $k$ and with bounded rank, the complexity of solving~\cref{prob PBO} on a signed hypergraph whose underlying hypergraph is $G_k$ grows super-polynomially in $k$ (see~\cref{{th PBO hard}}). 
    To prove this result, we first obtain an intractability result for BQO problems which is based on a complexity result for inference in graphical models~\cite{ChaSreHar08} (see~\cref{th treewidth}). Subsequently, using the inflation operation introduced in~\cite{dPKha23}, we present a polynomial-time reduction of BQO instances to pseudo-Boolean optimization instances, and as a result obtain an intractability result for pseudo-Boolean optimization.   
    Second, leveraging existing results on the extension complexity of the Boolean quadric polytope~\cite{AboFio19}, we obtain the first lower bound on the extension complexity of the pseudo-Boolean polytope. Namely, we prove that for 
    any hypergraph $G$ of bounded rank, there exists a signed hypergraph $H$ with the underlying hypergraph $G$, such that the extension complexity of the pseudo-Boolean polytope $\PBP(H)$ grows exponentially in the treewidth of $G$ (see~\cref{th xcPBP}).
    It is important to note that the bounded rank assumption is key for these results, as, for example, a polynomial-size extended formulation for the multilinear polytope of $\beta$-acyclic can be constructed in polynomial-time~\cite{dPKha23mMPA}, and a $\beta$-acyclic hypergraph with unbounded rank has an unbounded treewidth.

    %\item 
    \parag{Positive results}
    In Section~\ref{sec:sufficient}, we obtain new polynomial-size extended formulations for the pseudo-Boolean polytope of signed hypergraphs whose underlying hypergraphs contain $\beta$-cycles. To this end, we make use of the hypergraph-theoretic notion of nest-sets recently introduced in~\cite{lanz23}, which is a natural generalization of nest-points. We then introduce the notion of nest-set gap for a hypergraph, a quantity equal to zero if and only if the hypergraph is $\beta$-acylic.  The nest-set gap constitutes a notion of ``distance'' from the hypergaph acyclicity. We prove that if this distance is bounded, then the pseudo-Boolean polytope admits a polynomial-size extended formulation (see~\cref{{conj1}}). 
    The complexity of checking whether the nest-set gap of a hypergraph is bounded is unknown. However, checking the boundedness of a related quantity, namely, the nest-set width of a hypergraph, can be solved in polynomial time. 
    The nest-set width of a hypergraph is equal to one if and only if the hypergraph is $\beta$-acyclic. Moreover, the nest-set width of a hypergraph is lower bounded by its nest-set gap; this, in turn, implies that the pseudo-Boolean polytope of a signed hypergraph with bounded nest-set width has a polynomial-size extended formulation as well. Our proposed extended formulations for signed hypergraphs with bounded nest-set width can be constructed in polynomial-time, hence providing polynomial-time algorithms for a large class of pseudo-Boolean optimization problems whose underlying hypergraphs contain $\beta$-cycles.
 
    \medskip

    Figure~\ref{fig restrictions} summarizes various types of hypergraphs (resp. underlying hypergraphs) for which the existence of a polynomial-size extended formulation for the multilinear polytope (resp. the pseudo-Boolean polytope) is known.

\begin{figure}
  \centering

%\scalebox{1.0}{
\begin{tikzpicture}[node distance=0.7cm]

\node[dashed] (alpha) [startstop] {$\alpha$-acyclic};
\node (beta) [startstop, below=3 of alpha] {$\beta$-acyclic};
\node (gamma) [startstop, below=of beta] {$\gamma$-acyclic};
\node (Berge) [startstop, below=of gamma] {Berge-acyclic};

\node[ultra thick] (nsw) [startstop, above left=of beta] {bounded nest-set width};
\node[ultra thick] (nsg) [startstop, above=of nsw] {bounded nest-set gap};

\node (btwalpha) [startstop, above right=of beta] {bounded treewidth $\equiv$ \\ $\alpha$-acyclic with bounded rank};
\node (btw) [startstop, above=1 of btwalpha] {bounded incidence treewidth};

\draw[->, very thick]  (nsw) -- (nsg);
\draw[->, very thick]  (beta) -- (nsw);
\draw[->, very thick]  (Berge) -- (gamma);
\draw[->, very thick]  (gamma) -- (beta);
\draw[->, very thick]  (beta) -- (alpha);
\draw[->, very thick]  (btwalpha) -- (btw);
\draw[->, very thick]  (btwalpha) -- (alpha);

\end{tikzpicture}
%}

\caption{Hypergraph classes (resp. underlying hypergraph classes) for which the existence of a polynomial-size extended formulation for the multilinear polytope (resp. the pseudo-Boolean polytope) is known. The dashed lines depict the class for which the multilinear polytope (hence the pseudo-Boolean polytope) has exponential extension complexity. The solid lines depict the classes for which there exists a polynomial-size extended formulation for the pseudo-Boolean polytope (hence the multilinear polytope). Our contributions in this paper are depicted in thick solid lines. Arcs are directed from less general to more general. Properties with no directed connection are incomparable.
}
\label{fig restrictions}
\end{figure}

%\end{itemize}

\section{Intractability results for pseudo-Boolean optimization}
\label{sec:limits}

In this section, we provide the first intractability results, in the best-case sense, for pseudo-Boolean optimization. 
In the combinatorial optimization literature, it is well-understood that treewidth is a crucial metric in measuring the difficulty of many graph problems.
%; subsequently in Section~\ref{sec:poly}, \note{ADP} we relate these results with the class of polynomially solvable problems introduced in this paper. 
In the following, 
we denote by $\tw(G)$ the treewidth of $G$, where $G$ can be a graph, a hypergraph, or a signed hypergraph.
%for a graph $G$, we denote by $\tw(G)$ the treewidth of $G$.
%For a hypergraph $G$, we denote by $\tw(G)$ the treewidth of the intersection graph of $G$.
%For a signed hypergraph $H$, we denote by $\tw(H)$ the treewidth of the underlying hypergraph of $H$.
We also denote by $\poly(n)$, a polynomial function in $n$.
%\note{The next theorem is related to theorem~5.1 in \cite{ChaSreHar08} and theorem~7 in \cite{FaeMunPok22}.}
%\note{The result in \cite{ChaSreHar08} is for graphs, not hypergraphs.}
%Unbounded treewidth 0-1 QP is hard.
In our context, if the treewidth of a hypergraph is bounded, then the corresponding multilinear polytope admits a polynomial-size extended formulation~\cite{WaiJor04,Lau09,BieMun18}:

\begin{theorem}\label{th: alphaPoly}
Let $G=(V,E)$ be a hypergraph with $\tw(G) = k$. Then $\MP(G)$ has an extended formulation with $O(2^k |V|)$ variables and inequalities. Moreover, if $k \in O(\log \poly(|V|, |E|))$, then $\MP(G)$ has a polynomial-size extended formulation. 
\end{theorem}

\cref{th: alphaPoly} implies that if the hypergraph has a bounded treewidth, then~\cref{prob BPO} can be solved in polynomial-time. 
A similar result can be stated for the pseudo-Boolean polytope $\PBP(H)$ in terms of $\tw(H)$ (see theorem~6 in~\cite{dPKha23}).  Conversely, it is well-known that even for BQO problems, unbounded treewidth can imply intractability. In fact, the next theorem follows directly from the NP-hardness of the Simple Max Cut problem. 
\begin{theorem}[\cite{GarJohSto76}]
\label{th maxcut}
Let $K_n$ denote the complete graph with $n$ nodes.
Let $f$ be an algorithm that solves any instance of BQO on $K_n$ with coefficients in $\{0,\pm 1, \pm 2\}$ in time $T(n)$.
Then, assuming $\P \neq \NP$, $T(n)$ grows super-polynomially in $n$.
\end{theorem}

Observe that in a complete graph, the treewidth coincides with the number of nodes. 
In this section, we prove that for a pseudo-Boolean optimization problem on a signed hypergraph $H$, unbounded treewidth $\tw(H)$ \emph{always} implies (i) intractability of~\cref{prob PBO}, and (ii) exponential extension complexity of $\PBP(H)$.

\subsection{A necessary condition for polynomial-time solvability of pseudo-Boolean optimization}

In the following, leveraging existing works~\cite{ChaSreHar08,FaeMunPok22}, we first obtain an intractability result for BQO problems. Building upon this result, we then obtain intractability results for higher degree binary polynomial optimization problems. 
Given an instance $\Lambda$ of \cref{prob BPO} or \cref{prob PBO}, we refer to the \emph{size} of $\Lambda$ (also known as \emph{bit size}, or \emph{length}), denoted by $\|\Lambda\|$, as the number of bits required to encode it, which is the standard definition in the mathematical programming literature (see, for example \cite{SchBookIP,ConCorZamBook}).
Recall that the Bounded-error Probabilistic Polynomial time ($\BPP$) is the class of decision problems solvable by a probabilistic Turing machine in polynomial time with an error probability bounded by $1/3$ for all instances. 
Our intractability results in this section are under the assumption that $\NP \not\subseteq \BPP$.
It is widely believed that $\P=\BPP$, which implies that $\NP \not\subseteq \BPP$ is equivalent to $\P \neq \NP$.
For a precise definition of $\BPP$ and the commonly believed $\NP \not\subseteq \BPP$ hypothesis, we refer the reader to \cite{AroBar09}.

\subsubsection{Intractability of binary quadratic optimization}

%and the discussion is section~3.2 in \cite{FaeMunPok22}.
%\note{The next theorem is related to theorem~5.1 in \cite{ChaSreHar08} and theorem~7 in \cite{FaeMunPok22}.}

Our intractability result, which is closely related to theorem~5.1 
in~ \cite{ChaSreHar08} and theorem 7 in~\cite{FaeMunPok22}, implies that there exists no class of graphs with unbounded treewidth for which every BQO problem on these graphs can be solved in time polynomial in the treewidth. 
More precisely, for every sequence of graphs $\bra{G_k}_{k=1}^\infty$ indexed by the treewidth $k$, there exists a choice of objective function coefficients $c$ such that the runtime of any algorithm that solves the BQO problem is super-polynomial in the treewidth $k$. In theorem~5.1 
of~\cite{ChaSreHar08}, the authors prove a similar result for the inference problem on binary pairwise graphical models. It is known that this inference problem can be formulated as a BQO problem. However, we state an independent proof since first our proof is direct; \ie it is concerned with a BQO problem and not inference, and second theorem~5.1 
in~\cite{ChaSreHar08} is proved under the stronger assumption  $\NP \not\subseteq \P/ \poly$, where $\P/ \poly$
denotes the class of decision problems that can be solved by a polynomial-time Turing machine with advice strings of length polynomial in the input size. Unlike other polynomial-time classes such as $\P$ or $\BPP$, the class $\P/ \poly$, is considered impractical for computing.
In theorem 7 of~\cite{FaeMunPok22}, the authors prove an intractability result similar to ours for quadratically constrained quadratic optimization problems, and  their reduction arguments rely on the fact that the optimization problem has linear and quadratic constraints.  Despite these technical differences, our proof technique follows the general scheme first proposed in~\cite{ChaSreHar08} and later refined in~\cite{FaeMunPok22}.

We make use of the next theorem, which is essentially a consequence of the celebrated graph minor theorem~\cite{ChekChu16}.
%In the next theorems, we utilize an explicit polynomial function $\kappa(n) \in O(n^{98}\poly\log(n))$.
% \begin{theorem}[Corollary~12 in~\cite{FaeMunPok22}]\label{gminor 1}
% Let $\bar G$ be a planar graph with $n$ nodes. 
% Then, $\bar G$ is a minor of all graphs of treewidth at least $\kappa(n) \in O(n^{98}\poly\log(n))$.
% \end{theorem}
The proof of this theorem is stated inside the proof of theorem~7 in~\cite{FaeMunPok22}. We include it here for completeness. 

\begin{theorem}\label{gminor 2}
Any planar graph $\bar G$ with $n$ nodes is a minor of any graph $G$ with treewidth at least $\kappa(n) \in O(n^{98}\poly\log(n))$.
Furthermore, there is a randomized algorithm that, given $G$, outputs the sequence of minor operations transforming $G$ into $\bar G$ in time $O(\poly(|V(G)|\cdot\kappa(n)))$, with high probability.
\end{theorem}

\begin{proof}
Let $\bar G$ be a planar graph with $n$ nodes, and let $G$ be a graph with treewidth at least $\kappa(n) \in O(n^{98}\poly\log(n))$.
%By \cref{gminor 1}, $\bar G$ is a minor of $G$.
Since $\bar G$ is planar, it is a minor of the $n/c \times n/c$ grid, for some constant $c$ \cite{RobSeyThom94}, and the sequence of minor operations can be found in linear time \cite{TamTol89}.
It follows from the graph minor theorem in~\cite{ChekChu16} that the $n/c \times n/c$ grid is a minor of $G$, and that we can find the sequence of minor operations (with high probability) in $O(\poly(|V(G)|\cdot\kappa(n)))$ time.
\end{proof}

%\begin{remark}\label{deltaRed}
%    In~\cite{ChekChu16} the authors prove that~\cref{gminor 2} holds for any $\delta > \frac{1}{98}$. In~\cite{chuzhoy15} the author further improved this bound to $\delta > \frac{1}{19}$. However, her proof is non-constructive and does not provide an algorithm to find the sequence of minor operations transforming $G$ into $\bar G$. 
%\end{remark}

%Denote by $T_f(G, c)$ the running time of an Algorithm $f$ to solve an instance of BQO over a graph $G=(V,E)$ with objective coefficients $c \in \R^{V\cup E}$.  

%\begin{theorem}
%\label{th treewidth}
%Let $\bra{G_k}_{k=1}^\infty$ be an infinite sequence of graphs with $\tw(G_k)=k$ for all $k$.  
%Let $f=\bra{f_k}_{k=1}^\infty$ be any sequence of possibly non-uniform algorithms that solves the BQO with graph $G_k$.
%Let $T(k)$ denote the worst case running time of $f$ on instances of BQO with graph $G_k$ (\ie $T(k)=\max_{c} T_{f_k}(G_k, c)$). 
%Then, assuming $\NP \not\subseteq \BPP$, $T(k)$ grows super-polynomially in $k$.
%\end{theorem}

Henceforth, we say that a countable family of graphs $\bra{G_k}_{k=1}^\infty$ is \emph{polynomial-time enumerable} if a description of $G_k$ is computable in $\poly(k)$ time. 
This in turn implies that an encoding of $G_k$ of size polynomial in $k$ exists. The next theorem provides a necessary condition for polynomial-time solvability of BQO problems.

\begin{theorem}
\label{th treewidth}
Let $\bra{G_k}_{k=1}^\infty$ be a polynomial-time enumerable family of graphs with $\tw(G_k)=k$, for all $k$.
Let $f$ be an algorithm that solves any instance $\Lambda_k$ of BQO on graph $G_k$ in time $T(k) \cdot \poly(\|\Lambda_k\|)$.
Then, assuming $\NP \not\subseteq \BPP$, $T(k)$ grows super-polynomially in $k$. 
\end{theorem}

% \begin{theorem}
% \label{th treewidth 2}
% Let $\bra{G_k}_{k=1}^\infty$ be a polynomial-time enumerable family of graphs with $\tw(G_k)=k$, for all $k$.
% Let $f$ be an algorithm that solves any instance $\Lambda_k$ of a problem BQO on graph $G_k$ in time $T(k,\|\Lambda_k\|)$.
% %Assume that, for every fixed $k$, $T(k,\|\Lambda_k\|)$ is bounded by a polynomial function in $\|\Lambda_k\|$.
% Then, assuming $\NP \not\subseteq \BPP$, $T(k,\|\Lambda_k\|)$ is not bounded by any polynomial function in $k,\|\Lambda_k\|$.
% \end{theorem}

%\note{Let $f$ be an algorithm that solves any instance $\Lambda_k$ of a problem BQO on graph $G_k$ in time $T(\|\Lambda_k\|)$. Then, assuming $\NP \not\subseteq \BPP$, $T(\|\Lambda_k\|)$ grows super-polynomially in $k$.}

\begin{proof}
The problem \emph{Max 2-SAT} on a planar graph, also known as \emph{planar Max-2SAT,} is NP-complete~\cite{guibas93}. 
It can be checked that an instance of Max-2SAT on a planar graph $\bar G = (\bar V,\bar E)$ can be formulated as a BQO problem on the same graph $\bar G$, and with objective coefficients $\bar c_v \in \{\pm 1, 0\}$ for all $v \in \bar V$ and $\bar c_e \in \{\pm 2, \pm 1, 0\}$ for all $e \in \bar E$. 
Let $n := |\bar V|$.
By \cref{gminor 2}, $\bar G$ is a minor of $G_{\kappa(n)} = (V_{\kappa(n)}, E_{\kappa(n)})$, where $\kappa(n)$ is defined in the statement of the theorem.
Since by the enumerability assumption $V_{\kappa(n)}$ is bounded by a polynomial in $\kappa(n)$, by \cref{gminor 2}, there is a polynomial-time randomized algorithm that, given $G_{\kappa(n)}$, outputs the sequence of minor operations transforming $G_{\kappa(n)}$ into $\bar G$, with high probability.
%$G_{\kappa(n)}$ can be constructed from $\bar G$ in polynomial time (with high probability). 

We now show that, given an instance of a BQO problem on a planar graph $\bar G = (\bar V, \bar E)$ with $\bar c_v \in \{\pm 1, 0\}$ for all $v \in \bar V$ and $\bar c_e \in \{\pm 2, \pm 1, 0\}$ for all $e \in \bar E$, we can construct in polynomial time an equivalent instance of a BQO problem on $G_{\kappa(n)}$.
First, we observe that $G_{\kappa(n)}$ can be constructed in time bounded by a polynomial in $\kappa(n)$, due to our enumerability assumption.
It suffices to show the equivalence for a single minor operation; The result then follows from a repeated application of this technique. 
Let $H=(V_H, E_H)$ be a graph and denote by $H'=(V_{H'}, E_{H'})$ a graph obtained from $H$ after a single minor operation. Consider an instance $\Lambda'$ of a BQO problem on $H'$ with objective function coefficients $c'$. 
We show how to solve this instance $\Lambda'$ by solving an instance $\Lambda$ of a BQO problem on $H$ with objective function coefficients $c$.
Recall that there are three graph minor operations (see, for example~\cite{ChaSreHar08}):
%\begin{enumerate}[leftmargin=*]
%    \item 

\smallskip

    \emph{1. Node deletion}: Let $u \in V_H$ and denote by $E_u$ all edges of $H$ containing $u$.  
    Then the minor of $H$ obtained by deleting node $u$ is $H'= (V_H\setminus \{u\}, E_H \setminus E_u)$.
    In this case, the objective function coefficients $c$ of $\Lambda$ can be defined as follows: $c_u = 0$, $c_e =0$
    for all $e \in E_u$ and $c_p = c'_p$ for all $p \in V_H\setminus \{u\}$
    and $p \in E_H \setminus E_u$.

    \smallskip

%    \item 
    \emph{2. Edge deletion}: Let $f \in E_H$.
    Then the minor of $H$ obtained by deleting edge $f$ is $H'= (V_H, E_H \setminus \{f\})$. 
    In this case, the objective function coefficients $c$ of $\Lambda$ can be defined as follows: $c_f = 0$ and $c_p = c'_p$ for all $p \in V_H$
    and $p \in E_H \setminus \{f\}$.

%    \item 

\smallskip

    \emph{3. Edge contraction}: Let $\{u,v\} \in E_H$ and denote by $E_u$ all edges of $H$ containing $u$. Let $E'_u$ be set obtained by replacing $u$ with $v$ in every element of $E_u$ and subsequently dropping $\{v,v\}$ from it.
    Then the minor of $H$ obtained by contracting edge $\{u,v\}$ to node $v$ is $H'= (V_H \setminus\{u\}, E_H \cup E'_u \setminus E_u)$. 
    Define $M := \sum_{p \in V_{H'}\cup E_{H'}}{|c'_p|}$.
    In this case, the objective function coefficients $c$ of $\Lambda$ can be defined as follows: $c_u = -M$, $c_v = c'_v - M$, $c_{\{u,v\}} =2M$, $c_p = c'_p$ for all $p \in V_H \setminus \{u,v\}$
    and $p \in E_H \setminus E_u$, and $c_e = c'_{e \cup \{v\}\setminus \{u\}}$ for all $e \in E_u \setminus \{u,v\}$. To see that $\Lambda'$ is equivalent to $\Lambda$, first note that $\Lambda'$ is given by
    \begin{align*}
        & \max \sum_{w \in V_{H'}}{c'_w x_w}+\sum_{\{w,q\} \in E_{H'}}{c'_{w,q} x_w x_q}\\
        & {\rm s.t.} \quad x_w \in \{0,1\}, w \in V_{H'}.
    \end{align*}
    Then $\Lambda'$ can be equivalently solved by solving the following optimization problem on $H$:
    \begin{align}\label{lqpc}
        & \max \sum_{w \in V_{H}}{c''_w x_w}+\sum_{\{w,q\} \in E_{H}}{c''_{w,q} x_w x_q}\\
        & {\rm s.t.} \quad x_u = x_v\nonumber\\
        & \quad \quad \; x_w \in \{0,1\}, w \in V_H \nonumber,
    \end{align}
    where $c''_u = 0$, $c''_{\{u,v\}} = 0$, $c''_p = c'_p$ for all $p \in V_H \setminus \{u\}$
    and $p \in E_H \setminus E_u$, and $c''_e = c'_{e \cup \{v\}\setminus \{u\}}$ for all $e \in E_u \setminus \{u,v\}$. In order to reformulate Problem~\eqref{lqpc} as a BQO problem on $H$, it suffices to remove the constraint $x_u = x_v$ and instead subtract the penalty term $M(x_u - x_v)^2$ from the objective function. The equivalence then follows. 
%\end{enumerate}

\medskip

We have explained how to reduce (with high probability) a BQO problem on $\bar G = (\bar V, \bar E)$ to a BQO problem on $G_{\kappa(n)} = (V_{\kappa(n)}, E_{\kappa(n)})$.
The number of arithmetic operations performed in the reduction is clearly polynomially bounded; therefore it suffices to show that the size of the objective function coefficients of the constructed instance is polynomially bounded by the size of the original BQO instance.
Recall that the objective coefficients of the original instance are $\bar c_v \in \{\pm 1, 0\}$ for all $v \in \bar V$, and $\bar c_e \in \{\pm 2, \pm 1, 0\}$ for all $e \in \bar E$.
%By Theorem~\ref{gminor 2} and the above three operations, we can reduce this problem BQO in polynomial time (with high probability) to problem BQO on the graph $G_{\kappa(n)}$.
%such that $\tw(G_{\kappa(n)})=\kappa(n)$ for some polynomial $\kappa(n)$. 
Let us now consider the objective function coefficients of the constructed instance of the BQO problem, denoted by $c$, and observe that $c$ has integer components.
Let $M$ be the sum of the absolute values of all coefficients in the original instance, \ie $M := \sum_{v \in \bar V}{|\bar c_v|}+ \sum_{e \in \bar E}{|\bar c_e|}$ and observe that $M \le |\bar V| + 2 |\bar E|$.
%Note that $M$ is linear in $|V|$ and $|E|$.
Our enumerability assumption, together with $|\bar V| \le |V_{\kappa(n)}|$, $|\bar E| \le |E_{\kappa(n)}|$, imply that $M$ is polynomial in $\kappa(n)$ as well.
Now consider the three types of minor operations defined above.
First, note that the first two minor operations do not change the sum of the absolute values of all coefficients.
%; \ie under these operations each element of $\bar c$ is in $\{\pm 2, \pm 1, 0\}$. 
Let us consider the last minor operation; \ie the edge contraction. 
By construction, after each edge contraction, the sum of the absolute values of all coefficients doubles.
%largest (in absolute value) objective function coefficient equals twice the sum of all objective function coefficients before the contraction.
%For a planar graph $\bar G$ with $|V| \geq 3$, we have $|E| \leq 3|V|-6$. 
%Since each edge contraction decreases the number of the edges of a graph at least by one, the total number of edge contractions is upper bounded by $|E|$.
%After the first application of an edge contraction, the largest (in absolute value) objective function coefficient equals $2M$.
%, which is upper bounded by $8n$ since by definition $M = |V|+2|E| \leq 8n$.  
%For the next edge contraction; say contracting the edge $\{u',v'\}$ one should subtract the penalty term $M'(x_{u'}-x_{v'})^2$ from the objective function, where $M'= \sum_{v \in V}{|c_v|}+ \sum_{e \in E}{|c_E|} + M = 2M$. 
%Hence the largest coefficient of the objective function in absolute value will be $4M$. 
Since there are at most $|E_{\kappa(n)}|$ contractions,
%Employing this argument recursively and using the upper bound on the number of edge contractions, 
we conclude that in the constructed instance, the sum of the absolute values of all coefficients is at most $2^{|E_{\kappa(n)}|} M$, whose size, by our enumerability assumption, is polynomial in $\kappa(n)$.

%largest coefficient of the objective function of the problem BQO over $G_{\kappa(n)}$ will $2^{|E|-1} M$, which is in turn upper bounded by $c_{\max} := n 2^{3n}$, and this number can be stored in polynomially many bits as $\log c_{\max} = \log n + 3n$.

Now let us use Algorithm $f$ to solve the resulting BQO problem over $G_{\kappa(n)}$ in time $T(\kappa(n)) \cdot \poly(\|\Lambda_{\kappa(n)}\|)$ and hence equivalently solve the initial BQO problem over the graph $\bar G$ with $n$ nodes. 
We already showed that $\|\Lambda_{\kappa(n)}\|$ is upper bounded by a polynomial in $\kappa(n)$, which in turn is bounded by a polynomial in $n$. 
Assume, for a contradiction, that $T(k)$ is bounded by a polynomial in $k$.
This implies that $T(\kappa(n))$ is bounded by a polynomial in $\kappa(n)$, and therefore by a polynomial in $n$.
It then follows that planar MAX-2SAT$\in \BPP$, which contradicts the assumption that $\NP \not\subseteq \BPP$.
Therefore, $T(k)$ grows super-polynomially in $k$.
%\note{might have to fix this paragraph}
\end{proof}

Together with Theorem~\ref{th: alphaPoly}, \cref{th treewidth} suggests that a bounded treewidth is necessary and sufficient for polynomial-time solvability of BQO problems. However, notice that there is a gap between these two results. 
For every positive integer $k$, let $G_k = (V_k,E_k)$ be a graph with treewidth $k$ and with $n_k:=|V_k| \in \Theta (2^{\sqrt k})$, hence $k \in \Theta (\log^2 n_k)$. 
Clearly, $\bra{G_k}_{k=1}^\infty$ does not satisfy the assumption of Theorem~\ref{th: alphaPoly} because $k \not\in O(\log \poly(n_k))$. 
On the other hand, $\bra{G_k}_{k=1}^\infty$ is not a polynomial-time enumerable family because $n_k$ is not bounded by a polynomial in $k$ and hence~\cref{th treewidth} is not applicable. 
To the best of our knowledge, at the time of this writing, no tractability or intractability result is known for the regime $k \in \omega(\log n_k)$ and $k \in o(n_k^{1/c})$ for any constant integer $c$.

\subsubsection{Intractability of \cref{prob BPO} and~\cref{,prob PBO}}

Our next goal is to obtain necessary conditions for polynomial-time solvability of higher-degree binary polynomial optimization problems. 
%Henceforth, we say that a countable family of hypergraphs $\bra{G_k}_{k=1}^\infty$ is \emph{polynomial-time enumerable} if there exists a polynomial-time algorithm that, given $k$, outputs a description of the intersection graph of $G_k$. 
Henceforth, we say that a countable family of hypergraphs $\bra{G_k}_{k=1}^\infty$ is \emph{polynomial-time enumerable} if a description of $G_k$ is computable in $\poly(k)$ time. 
%Note that this does not imply that there is an encoding of $G_k$ of size polynomial  in $k$, as the number of edges of $G_k$ can be exponentially larger than the number of edges of its intersection graph.
%As an example, consider the hypergraph $G_k$ with $k$ nodes and all possible $2^k-1-k$ edges. While this hypergraph has exponentially many edges, the number of edges of its intersection graph is just $k(k-1)/2$.
%
A fairly straightforward application of~\cref{th treewidth} gives the following intractability result for~\cref{prob BPO}.

\begin{corollary}
\label{cor BPO hard}
Let $\bra{G_k}_{k=1}^\infty$ be a polynomial-time enumerable family of hypergraphs with $\tw(G_k)=k$. 
Suppose that for every two nodes $u,v$ contained in an edge of $G_k$, $\{u,v\}$ also is an edge of $G_k$ for all $k$.
Let $f$ be an algorithm that solves any instance $\Lambda_k$ of \cref{prob BPO} on hypergraph $G_k$ in time at most $T(k) \cdot \poly(\|\Lambda_k\|)$.
Then, assuming $\NP \not\subseteq \BPP$, $T(k)$ grows super-polynomially in $k$.
% Suppose there exists an algorithm that solves \cref{prob BPO} with a hypergraph in $\G$, in polynomial time.
% Then, assuming $\NP \not\subseteq \BPP$, all hypergraphs in $\G$ have primal treewidth bounded by a constant.
\end{corollary}

\begin{proof}
For every positive $k$, denote by $G'_k$ the intersection graph of $G_k$.
%Since the rank of $G_k$ is at most $r$, we have $|E'_k| \le |V_k|^r$. 
Note that, by assumption, each edge of $G'_k$ is also an edge of $G_k$.
Therefore, the family of intersection graphs $\bra{G'_k}_{k=1}^\infty$ is polynomial-time enumerable.
%Furthermore, for each $k$, each edge in the intersection graph of $G_k$ is also an edge of $G_k$.
The proof then follows from \cref{th treewidth}, by setting to zero the costs of all the edges of $G_k$ of cardinality greater than two. 
%Notice that since the rank of $G_k$ is upper bounded by the constant $r$, the number of edges of $G_k$ is at most $|V_k|^r$, which is upper bounded by a polynomial in $k$. 
\end{proof}

We observe that the inference problem on higher-order graphical models can be formulated as \cref{prob BPO} satisfying the assumption of~\cref{{cor BPO hard}}. Namely, the corresponding hypergraph $G_k$ has a fixed rank and therefore has polynomially many edges in $|V_k|$ (and hence in $k$). Moreover, for every edge of the hypergraph, all possible subsets of cardinality at least two are also in the edge set (see~\cite{KhaWang24} for a detailed derivation). 
Nonetheless,~\cref{cor BPO hard} is rather restrictive, as it does not consider general \emph{sparse} hypergraphs. Interestingly, as we show next, for \cref{prob PBO} on signed hypergraphs with a bounded rank, no such restrictive assumption is required. 

To prove our next intractability result, we make use of \emph{inflation operation} that was recently introduced in~\cite{dPKha23}.
Let $H=(V,S)$ be a signed hypergraph, let $s \in S$, and let $e \subseteq V$ such that $e_s \subset e$.
Denote by $I(s,e)$ the set of all possible signed edges $s'$ with $e_{s'} = e$ such that $\eta_s(v) = \eta_{s'}(v)$ for every $v \in e_s$. Define $S' := S \cup I(s,e) \setminus \{s\}$.
We say that $H'=(V,S')$ is obtained from $H$ by \emph{inflating} $s$ to $e$. 
%We also say that $H'$ is obtained from $H$ via an \emph{inflation operation}.

In the following, we say that a function $r(k)$ is log-poly if $2^{r(k)}$ is a polynomial function. 
Recall that a hypergraph $G'=(V',E')$ is a \emph{partial hypergraph} of a hypergraph $G=(V,E)$ if $V' \subseteq V$ and $E' \subseteq E$. We are now ready to state our intractability result for pseudo-Boolean optimization:

\begin{theorem}
\label{th PBO hard}
Let $\bra{G_k}_{k=1}^\infty$ be a polynomial-time enumerable family of hypergraphs with $\tw(G_k)=k$ for all $k$, and with rank $r(k)$ that is upper bounded by a log-poly function in $k$.
Let $f$ be an algorithm that solves any instance $\Lambda_k$ of \cref{prob PBO} on a signed hypergraph with the underlying hypergraph $G_k$, in time $T(k) \cdot \poly(\|\Lambda_k\|)$.
Then, assuming $\NP \not\subseteq \BPP$, $T(k)$ grows super-polynomially in $k$.
\end{theorem}

\begin{proof}
For every positive $k$, let $G_k = (V_k,E_k)$, and let $G'_k = (V_k,E'_k)$ be the intersection graph of $G_k$, which has treewidth $k$.
Note that $|E'_k| \le |V_k|^2$, thus $|E'_k|$ is bounded by a polynomial in $k$.
It follows that $\bra{G'_k}_{k=1}^\infty$ is a polynomial-time enumerable family of graphs with $\tw(G'_k)=k$ for all $k$.
In the remainder of the proof, we show that every instance $\Lambda'_k$ of problem BQO on graph $G'_k$ can be polynomially reduced to an instance $\Lambda_k$ of \cref{prob PBO} on a signed hypergraph $H_k$, whose underlying hypergraph is a partial hypergraph of $G_k$.
Since each obtained \cref{prob PBO} can be stated as a \cref{prob PBO} on a signed hypergraph with the underlying hypergraph $G_k$, by setting the costs of additional $\poly(k)$ signed edges to zero, the result then follows from~\cref{th treewidth}.
In the remainder of the proof, we consider one fixed $k$.

First, we define the signed hypergraph $H_k$.
By the definition of $G'_k$, for each edge $e \in E'_k$, there exists an edge $g \in E_k$ such that $e \subseteq g$; we denote by $g(e)$ one such edge of $E_k$.
We then define the signed hypergraph $H_k=(V_k,S_k)$ obtained from $G'_k$ by inflating each $e \in E'_k$ to $I(e,g(e))$, where we recall that $I(e,g(e))$ denotes the set of all possible signed edges $s$ with $e_s = g(e)$ such that $\eta_{s}(v) = \eta_e(v)$ for every $v \in e$.
Note that, technically, the inflation operation is defined only for signed hypergraphs, while here $G'_k$ is non-signed; however, we can think of $G'_k$ as signed, with all signs equal to $+1$.
%replacing each edge $e \in E'_k$ all possible signed edges parallel to $g(e)$.
Denote by $r_k$ the rank of $G_k$; we have $|g(e)| \le r(k)$, therefore the number of signed edges $s \in S_k$ with $e_s = g(e)$ that we introduce in each inflation operation is at most $2^{r(k)}$.
As a result, $|S_k| \le 2^{r(k)} \cdot |E'_k|$, thus $|S_k|$ is bounded by a polynomial in $k$.
By the definition of $H_k$, we then have $I(e,g(e)) \subseteq S_k$.

Denote by $c'_v$ for $v \in V_k$, and $c'_e$ for $e \in E'_k$, the objective function coefficients of the instance $\Lambda'_k$.
Next, we define the cost coefficients of the instance $\Lambda_k$, which we denote by $c_v$ for $v \in V_k$, and $c_s$ for $s \in S_k$.
We initialize $c_v := c'_v$ for every $v \in V_k$, and $c_s := 0$ for every $s \in S_k$.
Next, for each edge $e \in E'_k$, we update the cost coefficients of $\Lambda_k$ recursively as follows:
%Let $e$ be an edge of $G'_k$ and let $c'_e$ be the corresponding coefficient in the objective function of $\Lambda'_k$.
%If $(e,+)$ is a signed edge of $H_k$, we set to $c_k$ the corresponding coefficient in the objective function of $\Lambda_k$.
%If $(e,-)$ is a signed edge of $H_k$, we set to $-c_k$ the corresponding coefficient in the objective function of $\Lambda_k$.
For every $s \in I(e,g(e))$, we update $c_s := c_s + c'_e$.

To complete the proof, we observe that the instances $\Lambda'_k$ and $\Lambda_k$ are equivalent, in the sense that, for every feasible point $z_v \in \{0,1\}$, for $v \in V$, its objective value in $\Lambda'_k$ is equal to its objective value in $\Lambda_k$.
This claim follows because, for every $e \in E'_k$ the following equation holds: 
\begin{align*}
c'_e \prod_{v \in e} z_v
=
c'_e \prod_{v \in e} z_v \prod_{v \in g(e) \setminus e}(z_v + (1-z_v))
= 
c'_e \sum_{s \in I(e,g(e))} \prod_{v \in e_s} \sigma_s(z_v).
\end{align*}
\end{proof}

Note that \cref{th PBO hard} also holds if $r(k)$ is a constant greater than or equal to two.
Together with Theorem~\ref{th: alphaPoly}, \cref{th PBO hard} suggests that for pseudo-Boolean optimization on signed hypergraphs of log-poly rank, a bounded treewidth is a necessary and sufficient condition for tractability. It is important to note that the log-poly rank assumption is the key here, as otherwise such a conclusion is not valid (see, for example,~\cref{extended}). Similar to our discussion following \cref{th treewidth}, there is a subtle gap between these necessary and sufficient conditions; that is, to the best of our knowledge, for treewidth $k$ satisfying $k \in \omega(\log n_k)$ and $k \in o(n_k^{1/c})$ for any constant integer $c$,
no tractability or intractability result is known.

%Write equivalent definition of primal treewidth of the hypergraph.

% \begin{corollary}
% Beta-acyclic hypergraphs of rank $k$ have treewidth $k-1$.
% \end{corollary}

% \begin{corollary}
% Hypergraphs with nest-set width $w$ (constant) and rank $k$ (constant) have treewidth at most (constant).
% \end{corollary}

% \note{Question: Can we prove it directly, with an ``exact'' bound?}

% \begin{corollary}
% Hypergraphs with nest-set gap $g$ (constant) and rank $k$ (constant) have treewidth at most (constant).
% \end{corollary}

% \note{Question: Can we prove it directly, with an ``exact'' bound?}

%Note that this does not work for regular BPO, because we cannot use inflation.

%\subsection{Two open questions}

We conclude this section by posing two open questions; below, we give two statements which we do not know if they are true or false.
We say that a countable family of signed hypergraphs $\bra{H_k}_{k=1}^\infty$ is \emph{polynomial-time enumerable} if a description of $H_k$ is computable in $\poly(k)$ time. 
Our first statement is concerned with~\cref{prob PBO}:

\begin{statement}
\label{statement PBO harder}
Let $\bra{H_k}_{k=1}^\infty$ be a polynomial-time enumerable family of signed hypergraphs with $\tw(H_k)=k$ for all $k$, and with rank $r(k)$ that is upper bounded by a log-poly function in $k$.
Let $f$ be an algorithm that solves any instance $\Lambda_k$ of \cref{prob PBO} on signed hypergraph $H_k$, in time $T(k) \cdot \poly(\|\Lambda_k\|)$.
Then, assuming $\NP \not\subseteq \BPP$, $T(k)$ grows super-polynomially in $k$.
\end{statement}

If true, \cref{statement PBO harder} would constitute a stronger version of \cref{th PBO hard}.
In fact, while the thesis is the same, the assumptions of \cref{statement PBO harder} are much weaker than those of \cref{th PBO hard}.
Namely, for each $k$, in \cref{th PBO hard}, we assume that $f$ solves \cref{prob PBO} over instances on all signed hypergraphs with underlying hypergraph $G_k$.
In contrast, in \cref{statement PBO harder}, we assume that $f$ solves \cref{prob PBO} over instances on a single signed hypergraph $H_k$.
The idea used in the proof of \cref{th PBO hard} does not seem to be applicable to \cref{statement PBO harder}, since in the reduction we need to introduce a very specific set of signed edges.

%{\color{red} For Alberto: add a few lines explaining why this is stronger and why the previous proof idea doesn't apply.}

%Does \cref{prop treewidth algo} hold also for \cref{prob BPO}? 
%Or can we find a counterexample? 
Our second statement is similar to Theorem~\ref{th PBO hard}, but is concerned with~\cref{prob BPO}. 
It is important to note that the proof technique employed in Theorem~\ref{th PBO hard} is not applicable to prove the hardness for \cref{prob PBO}, as it uses the inflation operation, which involves the introduction of signed edges.
%More in detail, we ask whether the following statement is true or false.

\begin{statement}
\label{statement BPO hard}
Let $\bra{G_k}_{k=1}^\infty$ be a polynomial-time enumerable family of hypergraphs with $\tw(G_k)=k$ for all $k$, and with rank $r(k)$ that is upper bounded by a log-poly function in $k$.
Let $f$ be an algorithm that solves any instance $\Lambda_k$ of \cref{prob BPO} on hypergraph $G_k$, in time $T(k) \cdot \poly(\|\Lambda_k\|)$.
Then, assuming $\NP \not\subseteq \BPP$, $T(k)$ grows super-polynomially in $k$.
\end{statement}

Notice that if~\cref{statement PBO harder} is true, then~\cref{statement BPO hard} is also true. 
We remark that, even if instead of a log-poly function in $k$, we assume that the rank is a constant,  
\cref{statement PBO harder,statement BPO hard} remain open.

\subsection{Lower bounds on the extension complexity of the pseudo-Boolean polytope}

In this section, we show that a bounded treewidth is a necessary condition for the existence of a polynomial-size extended formulation for the pseudo-Boolean polytope of signed hypergraphs with log-poly ranks. 
Recall that given a polytope $\P$, 
%and an extended formulation for it $\Q$, 
%we define the \rm{size} of $\Q$ as the number of linear inequalities and equalities defining it. 
the \emph{extension complexity} of $\P$, denoted by $\xc(\P)$, is the minimum number of linear inequalities and equalities in an extended formulation of $\P$. 
Fiorini et al.~\cite{fiorini12} prove that over a complete graph, the Boolean quadric polytope has exponential extension complexity.

\begin{theorem}[\cite{fiorini12}]\label{fiorini}
    Let $K_n$ denote the complete graph with $n$ nodes. Then there exists a universal constant $C > 0$ such that for all $n$ we have $\xc(\BQP(K_n)) \geq 2^{Cn}$.
\end{theorem}

We then obtain the following result regarding the extension complexity of the multilinear polytope.

\begin{corollary}\label{xcMP1}
% There exists a universal constant $C > 0$ such that, for all $n$ and for all hypergraphs $G=(V,E)$ with $n$ nodes whose edge set contains all subsets of $V$ of cardinality two, we have $\xc(\MP(G)) \geq 2^{Cn}$.
For any hypergraph $G=(V, E)$, denote by $G'$ the graph $(V,E')$, where $E'$ contains all edges in $E$ of cardinality two.
Then there exists a universal constant $C > 0$ such that, for every hypergraph $G$, we have $\xc(\MP(G)) \geq 2^{Cn}$, where $n$ is the size of the largest clique in $G'$.
\end{corollary}

\begin{proof}
Since $E' \subseteq E$, any extended formulation for the multilinear polytope $\MP(G)$ also serves as an extended formulation for $\BQP(G')$.
Since $G'$ contains a clique with $n$ nodes, any extended formulation for $\BQP(G')$ also serves as an extended formulation for $\BQP(K_n)$. 
By~\cref{fiorini} we deduce $\xc(\MP(G)) \ge \xc(\BQP(G')) \ge 2^{Cn}$ for some constant $C > 0$.    
\end{proof}

From \cref{xcMP1} we deduce that there exist $\alpha$-acyclic hypergraphs for which the multilinear polytope has exponential extension complexity.

\begin{corollary}
For every positive integer $n$, there exists an $\alpha$-acyclic hypergraph $G$ with $n$ nodes and $O(n^2)$ edges such that $\xc(\MP(G)) \ge 2^{Cn}$ for some constant $C > 0$.
\end{corollary}

\begin{proof}
Follows from \cref{xcMP1}, by considering the $\alpha$-acyclic hypergraph obtained from the complete graph on $n$ nodes by adding one edge containing all nodes.
\end{proof}

In the following we show that the extension complexity of the pseudo-Boolean polytope is exponential in the tree-width of its signed hypergraph. To this end, we make use of the following theorem relating the extension complexity of the Boolean quadric polytope to the treewidth of its graph. This result follows from the proof of theorem 3 in~\cite{AboFio19} together with the graph minor theorem~\cite{ChekChu16}: 

%For completeness, we present an independent proof here. To this end, we make use of the following results. We denote by ${\rm STAB}(G)$ the stable set polytope on graph $G$.

%\begin{lemma}[\cite{AviTiw13}]\label{xcplanar}
%For every positive $n$ there exists a planar graph $G$ with $O(n^2)$ nodes and edges such that $\xc({\rm STAB}(G)) \geq 2^{\Omega(\sqrt{n})}$.
%\end{lemma}

%\begin{lemma}[\cite{AboFio19}]\label{xcminor}
%    Let the graph $H$ be a minor of the graph $G$.
%    Then $\xc(\BQP(G)) \geq \xc(\BQP(H))$.
%\end{lemma}

%We are now ready to present the result regarding the extension %complexity of the Boolean quadric polytope.

\begin{theorem}
\label{th xcBQP}
There exists a universal constant $\delta > \frac{1}{20}$ such that, for every graph $G$ with $n$ nodes, we have $\xc(\BQP(G)) \geq 2^{\Omega((\tw(G))^\delta+\log n)}$.
\end{theorem}
%\begin{proof}
%Let $k := \tw(G)$.
%By~\cref{gminor 2} and~\cref{deltaRed}, there is a universal constant $\delta > \frac{1}{19}$ such that for every $k \geq 1$, any graph $G$ with treewidth $k$ contains any planar graph $\bar G$ with $\Omega(k^\delta/\poly\log k)$ nodes, as a minor. By~\cref{xcminor}, we have $\xc({\rm BQP}(G)) \geq \xc({\rm BQP}(\bar G))$.
%By~\cref{xcplanar}, let $\bar G$ be a planar graph such that $\xc({\rm STAB}(\bar G)) \geq 2^{\Omega(k^{\frac{\delta}{4}})}$.
%The polytope ${\rm STAB}(\bar G)$ is a face of ${\rm BQP}(\bar G)$ obtained by setting $z_e = 0$ for all edges in $\bar G$. This in turn implies that $\xc({\rm BQP}(\bar G)) \geq \xc({\rm STAB}(\bar G))$.  Hence, we get $\xc({\rm BQP}(G)) \geq 2^{\Omega(k^{\frac{\delta}{4}})}$. This together with the fact that $\xc({\rm BQP}(G)) \geq n$, completes the proof.
%\end{proof}

By~\cref{th xcBQP}, the following result regarding the extension complexity of the multilinear polytope is immediate.

\begin{corollary}\label{xcMP2}
For any hypergraph $G=(V, E)$, denote by $G'$ the graph $(V,E')$, where $E'$ contains all edges in $E$ of cardinality two.
Then there exists a universal constant $\delta > \frac{1}{20}$ such that for every hypergraph $G$ with $n$ nodes, we have $\xc(\MP(G)) \geq 2^{\Omega((\tw(G'))^\delta+\log n)}$.
\end{corollary}

%\begin{theorem}
%\label{th xcBQP}
%Let $\bra{G_k}_{k=1}^\infty$ be a family of graphs with $\tw(G_k)=k$, for all $k$. Then there exists a universal constant $C$ such that $\xc(\BQP(G_k)) \geq 2^{k^{\frac{1}{C}}}$ for all $k$.
%\end{theorem}

%\begin{proof}
%    For a positive $n$, by~\cref{planarBQP}, let $\bar G =(\bar V, \bar E)$ denote a planar graph on $O(n^2)$ nodes such that $\xc(\BQP(\bar G)) \ge 2^{\Omega(\sqrt{n})}$. By \cref{gminor 2}, $\bar G$ is a minor of any graph $G_{\kappa(n)}$ with $\tw(G_{\kappa(n)}) = \kappa(n)$, where $\kappa(n) \in O(n^{196}\poly\log(n))$. By~\cref{th minor} we have 
%    $\xc(\BQP(G_{\kappa(n)})) \geq \xc(\BQP(\bar G))$. Since $\kappa(n)$ is bounded by a polynomial in $n$, by letting $k=\kappa(n)$ we obtain 
 %   $\xc(\BQP(G_{k})) \geq 2^{k^{\frac{1}{C}}}$ for some constant $C$.
  
%\end{proof}

Both~\cref{xcMP1} and~\cref{xcMP2} are rather restrictive, as they do not consider general \emph{sparse} hypergraphs. As we show next, thanks to the inflation operation, for the pseudo-Boolean polytope of signed hypergraphs with a log-poly rank, no such restrictive assumptions are required.
To this end, we make use of the following theorem that is concerned with the inflation operation. 
Let $H=(V,S)$ be a signed hypergraph, let $s \in S$, and let $e \subseteq V$ such that $e_s \subset e$.
Let $H'=(V,S')$ be obtained from $H$ by inflating $s$ to $e$. 
The following theorem indicates that if an extended formulation for $\PBP(H')$ is available, one can obtain an extended formulation for $\PBP(H)$ as well.

\begin{theorem} [Theorem~3 in~\cite{dPKha23}]
\label{lem inflation}
Let $H=(V,S)$ be a signed hypergraph, let $s \in S$, and let $e \subseteq V$ such that $e_s \subset e$.
Let $H'=(V,S')$ be obtained from $H$ by inflating $s$ to $e$.
Then an extended formulation of $\PBP(H)$ can be obtained by juxtaposing an extended formulation of $\PBP(H')$ and the equality constraint
\begin{equation}\label{eq inflation sum}
    z_s = \sum_{s' \in I(s,e)}{z_{s'}}.
\end{equation}
Moreover, if $\PBP(H')$ has a polynomial-size extended formulation and $|e|-|e_s| = O(\log\poly(|V|,|S|))$, then $\PBP(H)$ has a polynomial-size extended formulation as well.
\end{theorem}

We are now ready to state the main result of this section regarding the extension complexity of the pseudo-Boolean polytope.

\begin{theorem}
\label{th xcPBP}
There exists a universal constant $\delta > \frac{1}{20}$ such that, for every hypergraph $G$ with $n$ nodes whose rank is upper bounded by a log-poly function in $\tw(G)$, there exists a signed hypergraph $H$ with the underlying hypergraph $G$ for which we have $\xc(\PBP(H)) \geq 2^{\Omega((\tw(G))^\delta+\log n)}$.
\end{theorem}

\begin{proof}
Let $G = (V,E)$, and let $G' = (V,E')$ be the intersection graph of $G$, which has the treewidth $\tw(G') = \tw(G)$. 
%thus $|E'_k|$ is bounded by a polynomial in $k$.
By the definition of $G'$, for each edge $e \in E'$, there exists an edge $g \in E$ such that $e \subseteq g$; we denote by $g(e)$ one such edge of $E$.
We then define the signed hypergraph $\bar H=(V,\bar S)$ obtained from $G'$ by inflating each edge $e \in E'$ to $I(e,g(e))$, where we recall that $I(e,g(e))$ denotes the set of all possible signed edges $s$ with $e_s = g(e)$ such that $\eta_{s}(v) = \eta_e(v)$ for every $v \in e$.
%Note that, technically, the inflation operation is defined only for signed hypergraphs, while here $G'_k$ is non-signed; however, we can think of $G'_k$ as signed, with all signs equal to $+1$.
%replacing each edge $e \in E'_k$ all possible signed edges parallel to $g(e)$.
Denote by $r$ the rank of $G$; we have $|g(e)| \le r$, therefore the number of signed edges $s \in \bar S$ with $e_s = g(e)$ that we introduce in an inflation operation is at most $2^{r}$.
As a result, $|\bar S| \le 2^{r} \cdot |E'|$. Since $|E'| \le |V|^2$ and $r$ is a log-poly function in $\tw(G)$, we deduce that $|\bar S|$ is upper bounded by a polynomial in $\tw(G)$. Notice that the underlying hypergraph of $\bar H$ is a partial hypergraph of $G$. 

Now let $H = (V, S)$ be a signed hypergraph with the underlying hypergraph is $G$ such that $S \supseteq \bar S$. 
It then follows that an extended formulation for $\PBP(H)$ serves as an extended formulation for $\PBP(\bar H)$ as well, implying that $\xc(\PBP(H)) \geq \xc(\PBP(\bar H))$. Now, let us consider $\PBP(\bar H)$.  By~\cref{lem inflation} an extended formulation for ${\rm BQP}(G')$ is given by an extended formulation for $\PBP(\bar H)$ together with $|E'|$ inequalities containing $|\bar S|$ additional variables. Since both
$|E'|$ and $|\bar S|$ are upper bounded by a polynomial in $\tw(G)$, and since by~\cref{th xcBQP}, $\xc(\BQP(G')) \geq 2^{\Omega((\tw(G))^\delta+\log n)}$ for some constant $\delta > \frac{1}{20}$, we deduce that $\xc(\PBP(\bar H)) \geq 2^{\Omega((\tw(G))^\delta+\log n)}$. This in turn implies that $\xc(\PBP(H)) \geq 2^{\Omega((\tw(G))^\delta+\log n)}$.
\end{proof}

We conclude this section by posing two open questions regarding the extension complexity of the pseudo-Boolean polytope and the multilinear polytope; below, we give two statements that we do not know if they are true or false. The first statement constitutes a stronger version of~\cref{th xcPBP}, while the second statement is concerned with the multilinear polytope.

\begin{statement}
\label{st ext 2}
There exists a universal  positive constant $\delta$ such that, for every signed hypergraph $H$ with $n$ nodes whose rank is upper bounded by a log-poly function in $\tw(H)$, we have $\xc(\PBP(H)) \geq 2^{\Omega((\tw(H))^\delta+\log n)}$.
\end{statement}

\begin{statement}
\label{st ext 1}
There exists a universal positive constant $\delta$ such that, for every hypergraph $G$ with $n$ nodes whose rank is upper bounded by a log-poly function in $\tw(G)$, we have $\xc(\MP(G)) \geq 2^{\Omega((\tw(G))^\delta+\log n)}$.
\end{statement}

Clearly, if \cref{st ext 2} is true, then \cref{st ext 1} is true as well. Even if instead of a log-poly function, we assume that the rank is a constant, \cref{st ext 2,st ext 1} remain open.

\section{Tractability beyond hypergraph acyclicity for pseudo-Boolean optimization}
\label{sec:sufficient}

In this section, we move beyond hypergraph acyclicity and obtain new sufficient conditions under which the pseudo-Boolean polytope admits a polynomial-size extended formulation. 
%Our proofs are constructive and the proposed extended formulations can be obtained in polynomial-time, thereby translating into polynomial-time algorithms to solve general classes of binary polynomial optimization problems.
In a recent work~\cite{dPKha23}, the authors introduce the pseudo-Boolean polytope, a generalization of the multilinear polytope that enables them to unify and extend all prior results on the existence of polynomial-size extended formulations for the convex hull of the feasible region of binary polynomial optimization problems of degree at least three. In particular, given a signed hypergraph $H=(V,S)$, the authors prove that if the underlying hypergraph of $H$ is $\beta$-acyclic, then $\PBP(H)$ has an extended formulation of size polynomial in $|V|,|S|$.
Recall that a \emph{$\beta$-cycle} of length $\ell$ for some $\ell \ge 3$ in a hypergraph $G$ is a sequence $v_1, e_1, v_2, e_2, \dots , v_\ell, e_\ell, v_1$ such that $v_1, v_2, \dots , v_\ell$ are distinct nodes, $e_1, e_2, \dots , e_\ell$ are distinct edges, and $v_i$ belongs to $e_{i-1}, e_i$ and no other $e_j$ for all $i = 1, \dots , \ell$, where $e_0 = e_\ell$.
A hypergraph is \emph{$\beta$-acyclic} if it does not contain any $\beta$-cycles.

\begin{theorem}\label{extended} \cite{dPKha23}
Let $H=(V,S)$ be a signed hypergraph of rank $r$ whose underlying hypergraph is $\beta$-acyclic. Then the pseudo-Boolean polytope $\PBP(H)$ has a polynomial-size extended formulation with at most $O(r|S||V|)$ variables and inequalities.
Moreover, all coefficients and right-hand side constants in the extended formulation are $0,\pm 1$.
\end{theorem}

Theorem~\ref{extended} is a generalization of the earlier result in~\cite{dPKha23mMPA} stating that the multilinear polytope of a $\beta$-acyclic hypergraph has a polynomial-size extended formulation.  Namely, Theorem~\ref{extended}
only requires the $\beta$-acyclicity of the underlying hypergraph of $H$. Indeed, the multilinear hypergraph of $H$, \ie $\mh(H)$ may contain many $\beta$-cycles. 
In this section, we present a significant generalization of Theorem~\ref{extended} that gives polynomial-size extended formulations for the pseudo-Boolean polytope of a large class of signed hypergraphs whose underlying hypergraph contains $\beta$-cycles. 
To define this class of hypergraphs, we make use of the notion of ``gap'' introduced in~\cite{dPKha23}. In this paper, we use this concept in a slightly more general form, which we define next.
%Consider a hypergraph $G=(V,E)$, and let $E' \subseteq E$; we define the \emph{gap} of $E'$ as
%$$
%\gap(E') := \max\Big\{ |\cup_{f \in E'} f| - |e| : e \in E' \Big\}.
%$$

Consider a hypergraph $G=(V,E)$, and let $V' \subseteq V$; we define the \emph{gap} of $G$ induced by $V'$ as
\begin{equation}\label{Hgap}
\gap(G, V') := \max\Big\{ |V'| - |e \cap V'| : e \in E , \ e \cap V' \neq \emptyset \Big\}.
\end{equation}
Moreover, for a signed hypergraph $H=(V,S)$, and $V' \subseteq V$, the gap of $H$ induced by $V'$, denoted by $\gap(H, V')$ is defined as the gap of the underlying hypergraph of $H$ induced by $V'$. 
In~\cite{dPKha23} the authors prove that thanks to the inflation operation, if $\gap(H,V) = O(\log\poly(|V|,|S|))$, the pseudo-Boolean polytope $\PBP(H)$ has a polynomial-size extended formulation.

\subsection{Nest-set, nest-set width, and nest-set gap}

In the following, we introduce some new hypergraph theoretic notions that enable us to obtain new polynomial-size extended formulations for the pseudo-Boolean polytope.
The proof of Theorem~\ref{extended} relies on the key concept of nest points. 
Let $G=(V,E)$ be a hypergraph. A node $v \in V$ is a \emph{nest point} of $G$ if the set of the edges of $G$ containing $v$ is totally ordered with respect to inclusion. It is simple to see that nest points can be found in polynomial time.
Let $N \subset V$.
We define the hypergraph obtained from $G$ by \emph{deleting} $N$ as the hypergraph $G-N$ with set of nodes $V \setminus N$ and set of edges $\{e \setminus N : e \in E, \ |e \setminus N| \ge 2\}$. 
When $N = \{v\}$, we write $G-v$ instead of $G-N$.
A \emph{nest point elimination order} is an ordering $v_1, \dots, v_n$ of the nodes of $G$, such that $v_1$ is a nest point of $G$, $v_2$ is a nest point of $G - v_1$, and so on, until $v_n$ is a nest point of $G - v_1 - \dots - v_{n-1}$.
%We can write this condition compactly as $v_i$ is a nest point of $G - v_1 - \dots - v_{i-1}$, for $i\in [s]$, where we make the slight abuse of notation $G - v_1 - \dots - v_0 = G$.
The following result provides a characterization of $\beta$-acyclic hypergraphs, in terms of nest points:
\begin{theorem}[\cite{Dur12}]\label{th beta iff}
A hypergraph $G$ is $\beta$-acyclic if and only if it has a nest point elimination order.
\end{theorem}
Recently, in the context of satisfiability problems, Lanzinger~\cite{lanz23} generalizes the concept of nest points  to~\emph{nest-sets}, which we define next. Let $G=(V,E)$ be a hypergraph and let $N \subseteq V$. 
%Denote by $F(N)$ the set of edges in $E$ containing some $v \in N$. 
We say that $N$ is a \emph{nest-set} of $G$ if the set 
\begin{equation}\label{nestsetminus}
%F\setminus N :=
\{e \setminus N: e \in E, \ e \cap N \neq \emptyset\},
\end{equation}
is totally ordered with respect to inclusion. 
Clearly if $|N|=1$, then $N$ consists of a nest point of $G$. Let $N_1,N_2, \dots,N_t$ be pairwise disjoint subsets of $V$ such that $\cup_{i \in [t]}{N_i}=V$. 
We say that $\N= N_1, \cdots, N_t$  is a \emph{nest-set elimination order} of $G$, if $N_1$ is a nest-set of $G$, $N_2$ is a nest-set of $G-N_1$, and so on, until $N_t$ is nest-set of $G - N_1 - \dots - N_{t-1}$. 

Given a nest-set elimination order $\N$ of $G$, we define \emph{nest-set width} of $\N$, denoted by ${\rm nsw}_{\N}(G)$, as the maximum cardinality of any element in $\N$.  
%Define 
%$$F \cap N:=\{e \cap N: e \in F(N)\}.$$ 
We then define \emph{nest-set gap} of $\N$ as
%$$
%{\rm nsg}_{\N}(G) := \max\{\gap(F \cap N): N\in \N\}.
%$$
$$
{\rm nsg}_{\N}(G) := \max\Big\{\gap(G-N_1-\cdots-N_{i-1}, N_i):i \in [t]\Big\},
$$ 
where $\gap(\cdot,\cdot)$ is defined by~\eqref{Hgap}.
From these definitions, it follows that 
\begin{equation}\label{relation}
{\rm nsg}_{\N}(G) \leq {\rm nsw}_{\N}(G) -1.
\end{equation} 
In fact, ${\rm nsg}_{\N}(G)$
can be much smaller than ${\rm nsw}_{\N}(G)$. The following example demonstrates this fact.

\begin{example}\label{secex}
Consider a hypergraph $G=(V,E)$ whose edge set $E$ consists of all subsets of $V$ of cardinality $|V|-1$.
Letting $\N=V\setminus \{\bar v\}, \{\bar v\}$ for some $\bar v \in V$, it follows that ${\rm nsw}_{\N}(G) = |V|-1$,
while ${\rm nsg}_{\N}(G) = 1$. 
\end{example}

We define the nest-set width of $G$, denoted by ${\rm nsw}(G)$, as the minimum value of ${\rm nsw}_{\N}(G)$ over all nest-set elimination orders $\N$ of $G$. 
Similarly, we define the nest-set gap of $G$, denoted by ${\rm nsg}(G)$, as the minimum value of ${\rm nsg}_{\N}(G)$ over all nest-set elimination orders $\N$ of $G$. Consider the hypergraph of~\cref{secex}. It can be checked that  ${\rm nsw}(G) = |V|-1$,
while ${\rm nsg}(G) = 1$.
We define the nest-set gap of a signed hypergraph as the nest-set gap of its underlying hypergraph.
Similarly, we define the notions of nest point, nest-set, and nest-set width, also for signed hypergraphs. 
See \cref{figure2,figure4,figure5} for an illustration of nest-sets, nest-set widths, and nest-set gaps.

\begin{figure}
\centering
\includegraphics[scale=0.6]{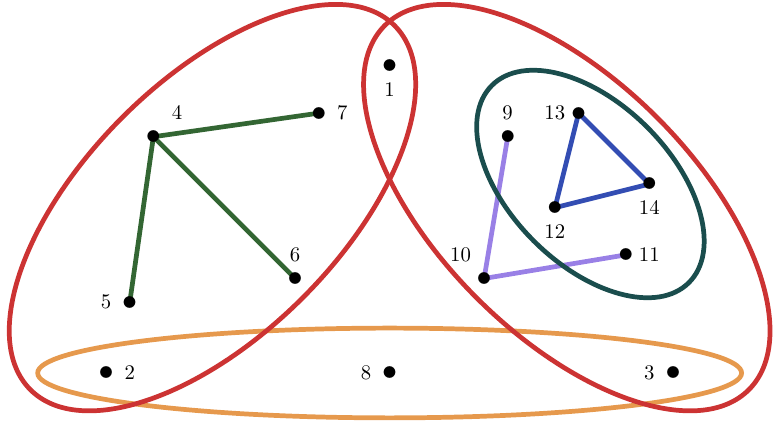}
\caption{A hypergraph $G$ with ${\rm nsw}(G) = 2$ and ${\rm nsg}(G) = 1$. A nest-set elimination order of $G$ is given by $\N=\{8\},\{7\}, \{6\}, \{5\}, \{4\}, \{12,13\}, \{14\},\{9,10\}, \{11\}, \{2,3\}, \{1\}$. $G$ contains $\beta$-cycles of length three.}
\label{figure2}
\end{figure}

\begin{figure}
\centering
\includegraphics[scale=0.6]{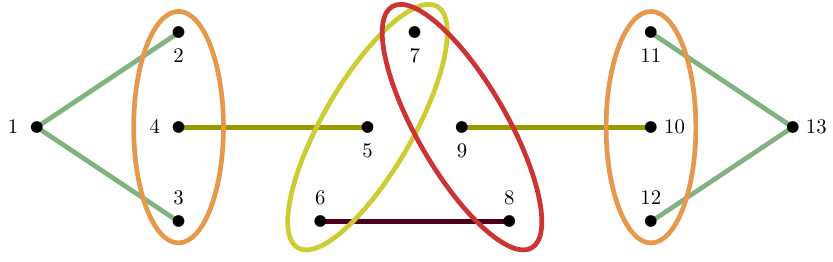}
\caption{A hypergraph $G$ with ${\rm nsw}(G) = 2$ and ${\rm nsg}(G) = 1$. A nest-set elimination order of $G$ is given by $\N=\{1,2\},\{3\}, \{4\}, \{5\}, \{6,7\},\{8\}, \{9\},\{10\}, \{11,12\},\{13\}$. $G$ contains $\beta$-cycles of length three.}
\label{figure4}
\end{figure}

\begin{figure}
\centering
\includegraphics[scale=0.6]{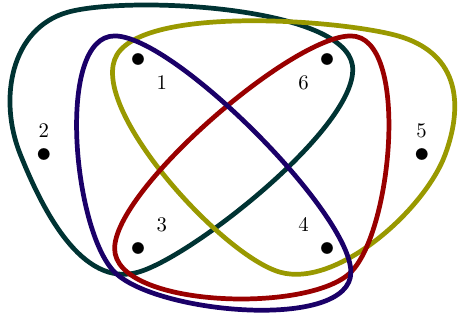}
\caption{A hypergraph $G$ with ${\rm nsw}(G) = 3$ and ${\rm nsg}(G) = 1$. A nest-set elimination order of $G$ is given by $\N=\{2\}, \{5\}, \{1,3,4\}, \{6\}$. $G$ contains $\beta$-cycles of length three.}
\label{figure5}
\end{figure}

From~\cref{th beta iff} it follows that ${\rm nsw} (G) = 1$, if and only if $G$ is $\beta$-acyclic; that is, hypergraphs with
${\rm nsw} (G) \geq  2$ contain $\beta$-cycles.
Hence, a natural question is whether ${\rm nsw} (G)$ is in general related to the length of $\beta$-cycles of $G$. 
In lemma~6 in~\cite{lanz23}, the author proves that if $G$ has a $\beta$-cycle of length $\ell$, then ${\rm nsw} (G) \geq \ell-1$.
Next, we strengthen this lower bound on the nest-set width: 

\begin{proposition}
\label{prop nest-set cycle propagation}
Let $C=v_1, e_1, v_2, e_2, \ldots, v_{\ell}, e_{\ell}, v_1$ be a $\beta$-cycle of length $\ell$ in a hypergraph $G$. 
Let $U_1 := (e_{\ell} \cap e_1) \setminus \cup_{j=2}^{\ell-1} e_j$.
Similarly, for $i=2,\dots,\ell$, let $U_i := (e_{i-1} \cap e_i) \setminus \cup_{j \in [\ell] \setminus \{i-1,i\}} e_j$.
For every nest-set $N$ of $G$ we have that $N \cap (\cup_{j=1}^\ell U_j)$ is the empty set, or contains at least $\ell-1$ sets among $U_j$, for $j \in [\ell]$.
Furthermore, ${\rm nsw} (G) \geq \min_{k \in [j]} \sum_{j \in [\ell] \setminus k} |U_j|$.
\end{proposition}

\begin{proof}
Let $U := \cup_{j=1}^\ell U_j$.
Suppose that $N \cap U$ is nonempty. 
That is, at least one node in $U$ is in $N$. 
Since we can rotate the indices of a cycle arbitrarily, we assume, without loss of generality, that this is a node $u_1 \in U_1$. 
Then $e_1$ and $e_{\ell}$ are both in $I(N)$, where we denote by $I(N)$ the set of edges of $G$ that contain at least one node in $N$.
By the definition of $U_2$, we have $U_2 \subseteq e_1 \setminus e_{\ell}$.
Similarly, $U_\ell \subseteq e_{\ell} \setminus e_1$. 
Thus, $e_1 \setminus N$ and $e_{\ell} \setminus N$ can only be comparable with respect to inclusion if $U_2$ or $U_{\ell}$ is contained in $N$. Suppose, without loss of generality, $U_2 \subseteq N$.

Then we have $e_2$ and $e_{\ell}$ in $I(N)$ and the same argument can be applied again, as long as the two edges are not adjacent in the cycle, and we find that $U_3$ or $U_{\ell}$ is contained in $N$, without loss of generality, $U_3 \subseteq N$.
We then obtain $U_2, U_3, \dots, U_{\ell-1} \subseteq N$.

Next, we consider the edges $e_1$ and $e_{\ell-1}$, which are both in $I(N)$. 
We have $U_1 \subseteq e_1 \setminus e_{\ell-1}$ and $U_\ell \subseteq e_{\ell-1} \setminus e_1$. 
Thus, $e_1 \setminus N$ and $e_{\ell-1} \setminus N$ can only be comparable with respect to inclusion if $U_1$ or $U_{\ell}$ is contained in $N$.
In either case, $N$ contains at least $\ell-1$ sets among $U_j$, for $j \in [\ell]$.

Since any node in the $\beta$-cycle $C$ has to be removed at some point in any nest-set elimination order, we obtain ${\rm nsw} (G) \geq \min_{k \in [j]} \sum_{j \in [\ell] \setminus k} |U_j|$.
\end{proof}

Notice that if the edge set of hypergraph $G$ of \cref{prop nest-set cycle propagation} is $E=\{e_1,\cdots, e_\ell\}$, then the bound given in this lemma is sharp; \ie  ${\rm nsw} (G) = \min_{k \in [j]} \sum_{j \in [\ell] \setminus k} |U_j|$. The next example demonstrates that the lower bound given in \cref{prop nest-set cycle propagation} is not sharp, in general.

\begin{example}\label{onlyex}
Consider a graph $G=(V,E)$ with $V=\{v_1,\cdots,v_n\}$ for some $n \geq 5$ and $E=\{\{v_1, v_i\}, \{v_2, v_i\}, \forall i \in \{3,\cdots,n\}\}$. See Figure~\ref{figure3} for an illustration.
It can be checked that $G$ contains cycles of length four only hence the lower bound on ${\rm nsw} (G)$ given by~\cref{prop nest-set cycle propagation} is three. However, ${\rm nsw} (G) = n-1 > 3$ since by assumption $n \geq 5$. To see this, first note that either $v_1$ or $v_2$ should be in a nest set, as otherwise set~\eqref{nestsetminus} will contain $\{v_1\}$ and $\{v_2\}$ and hence will not be totally ordered. Without loss of generality, suppose that $v_1$ is in the nest set. Second, note that $v_i$, for all $i \in \{3,\cdots,n\}$ should also be in the nest set as otherwise set~\eqref{nestsetminus} will contain $\{v_2\}$ and $\{v_j\}$ for some $j \in \{3,\cdots,n\}$ and hence will not be totally ordered. Therefore, ${\rm nsw} (G) = n-1$.
\end{example}

\begin{figure}
\centering
\includegraphics[scale=0.6]{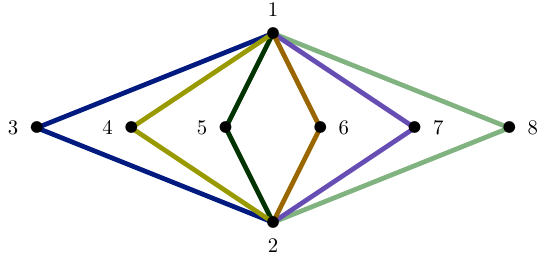}
\caption{An illustration of the graph of~\cref{onlyex} with $n=8$. This graph contains cycles of length four only, however, its nest-set width is seven.}
\label{figure3}
\end{figure}

A future direction of research is to investigate whether it is possible to give a characterization of nest-set width in terms of the ``size'' of the $\beta$-cycles in the hypergraph.

\medskip

As we detail in the next section, if the nest-set gap of the signed hypergraph is bounded, then the pseudo-Boolean polytope has a polynomial-size extended formulation.
Hence, we would like to understand the complexity of checking whether the nest-set gap of a hypergraph is bounded. Notice that by~\eqref{relation}, if the nest-set width is bounded, so is the nest-set gap.
The complexity of checking the nest-set width of a hypergraph was settled in~\cite{lanz23}. Lanzinger proves that deciding whether ${\rm nsw} (G) \leq k$ for any integer $k$ is NP-complete (see Theorem~10 in~\cite{lanz23}). However, when parameterized by $k$, this problem is fixed-parameter tractable:

\begin{proposition}[Theorem~15 in~\cite{lanz23}]\label{checknsw}
There exists a $2^{O(k^2)}\poly(|V|,|E|)$ time algorithm that takes as input hypergraph $G=(V,E)$ and integer $k \geq 1$ and returns a nest-set elimination order $\N$ with ${\rm nsw}_{\N}(G) =k$ if one exists, or rejects otherwise.
\end{proposition}

As~\cref{secex} indicates, there exist hypergraphs $G = (V,E)$ with
large nest-set width (e.g. ${\rm nsw}(G) = \theta(|V|)$) but with
fixed nest-set gap. At the time of this writing, the complexity of checking whether the nest-set gap of a hypergraph is bounded is an open question.

\subsection{The pseudo-Boolean polytope of signed hypergraphs with bounded nest-set gaps}
\label{sec:poly}

In this section, we obtain polynomial-size extended formulations for the pseudo-Boolean polytope of signed hypergraphs whose underlying hypergraphs contain $\beta$-cycles. 
Recall that the nest-set gap is zero, if and only if the hypergraph is $\beta$-acyclic.
We argue that the nest-set gap constitutes a notion of ``distance'' from the hypergaph acyclicity and we prove that if this distance is bounded, then the pseudo-Boolean polytope admits a polynomial-size extended formulation.
Our result serves as a significant generalization of~\cref{extended}. 

To obtain the new extended formulation, we make use of two tools developed in~\cite{dPKha23}. 
The first result is concerned with \emph{decomposability} of the pseudo-Boolean polytope $\PBP(H)$. 
Let $H=(V,S)$ be a signed hypergraph and let $V_1,V_2 \subseteq V$ such that $V = V_1 \cup V_2$, let $S_1 \subseteq \{s \in S : e_s \subseteq V_1\}$, $S_2 \subseteq \{s \in S : e_s \subseteq V_2\}$ such that $S = S_1 \cup S_2$.
Let $H_1 := (V_1,S_1)$ and $H_2 := (V_2,S_2)$. We say that $\PBP(H)$ is \emph{decomposable} into $\PBP(H_1)$ and $\PBP(H_2)$, if the system comprised of a description of $\PBP(H_1)$ and a description of $\PBP(H_2)$, is a description of $\PBP(H)$.
%The following theorem states that if the underlying hypergraph of $H$ has a nest point, then
%$\PBP(H)$ is decomposable to simpler polytopes.
%Let $G=(V,E)$ be a hypergraph. A node $v \in V$ is a \emph{nest point} of $G$ if the set of the edges of $G$ containing $v$ is totally ordered with respect to inclusion.

Consider a signed hypergraph $H = (V,S)$, let $s=(e_s, \eta_s) \in S$, 
%and let $U \subseteq V$. 
%In the following, when \note{fix here} we write $s \in U$, we mean $e \in U$. 
and let $v \in e_s$.
In the following, we denote by $s - v$ the signed edge $s'=(e_{s'},\eta_{s'})$, where $e_{s'}:= e_s \setminus \{v\}$, and $\eta_{s'}$ is the restriction of $\eta_s$ that assigns to each $v \in e_{s'}$ the sign $\eta_{s'}(v) = \eta_s(v)$.
The next theorem provides a sufficient condition for decomposability of the pseudo-Boolean polytope.

\begin{theorem}[Theorem~1 in~\cite{dPKha23}]
\label{th decomp}
Let $H = (V,S)$ be a signed hypergraph, and assume that it has a nest point $v$.
Let $s_1, s_2, \dots, s_k$ be the signed edges of $H$ whose underlying edges contain $v$, and assume $e_{s_1} \subseteq e_{s_2} \subseteq \cdots \subseteq e_{s_k}$.
Assume that $S$ contains the signed edges $s_i - v$ such that $|e_{s_i} - v| \ge 2$, for every $i \in [k]$.
Then $\PBP(H)$ is decomposable into $\PBP(H_1)$ and $\PBP(H_2)$, where $H_1$ and $H_2$ are defined as follows.
$H_1 := (V_1, S_v \cup P_v)$, where $V_1$ is the underlying edge of $s_k$, $S_v := \{s_1,\dots,s_k\}$, $P_v := \{s_i - v : |e_{s_i} - v| \ge 2, \ i \in [k]\}$, and $H_2 := H - v$.
\end{theorem}

The second result provides a polynomial-size extended formulation for the pseudo-Boolean polytope of a special type of signed hypergraphs, which we refer to as ``pointed.'' 
Consider a signed hypergraph $H=(V,S)$ and let $v \in V$ be a nest point of $H$. 
Denote by $S_v$ the set of all signed edges $s \in S$ such that $v \in e_s$.
Define $P_v:=\{s-v: s \in S_v, |e_s|\geq 3\}$.
We say  that $H$ is \emph{pointed} at $v$ (or is a \emph{pointed signed hypergraph}) if $V$ coincides with the underlying edge of the signed edge of maximum cardinality in $S_v$ and $S = S_v \cup P_v$.

\begin{theorem}[Theorem~2 in~\cite{dPKha23}]
\label{pointedHull}
Let $H=(V,S)$ be a pointed signed hypergraph. 
Then the pseudo-Boolean polytope $\PBP(H)$ has a polynomial-size extended formulation with at most $2|V|(|S|+1)$ variables and at most $4(|S|(|V|-2)+|V|)$ inequalities. Moreover, all coefficients and right-hand side constants in the system defining 
$\PBP(H)$ are $0,\pm 1$.
\end{theorem}

We are now ready to state the main result of this section.

\begin{theorem}\label{conj1}
Let $H=(V,S)$ be a signed hypergraph of rank $r$ whose underlying hypergraph $G=(V,E)$ satisfies ${\rm nsg}(G) \leq k$. Then the pseudo-Boolean polytope $\PBP(H)$ has an extended formulation with $O(r 2^k|V||S|)$ variables and inequalities.
In particular, if $k \in O(\log \poly(|V|, |E|))$, then $\PBP(H)$ has a polynomial-size extended formulation. Moreover, all coefficients and right-hand side constants in the system defining $\PBP(H)$ are $0,\pm 1$.
\end{theorem}

\begin{proof}
Denote by $\N= N_1, \cdots, N_t$ a nest-set elimination order of $G$ with ${\rm nsg}_{\N}(G) \leq k$. 
Consider the nest-set $N_1$. 

Assume $N_1=\{v\}$ for some $v \in V$.
Then, $v$ is a nest point of $G$. 
Denote by $S_{v}$ the set of all signed edges $s$ of $H$ with $v \in e_s$. 
By definition, the set $S_{v}$ is totally ordered.
Define the signed hypergraph $H'_1:=(V,S_v \cup P_{v})$, where $P_{v}:=\{s-v: s\in S_{v}, |e_s| \geq 3\}$.  
Clearly, an extended formulation for $\PBP(H'_1)$ serves as an extended formulation for $\PBP(H)$ as well.
Now define the pointed signed hypergraph $H_{v}:= (V_1, S_{v}\cup P_{v})$, where $V_1$ denotes the underlying edge of a signed edge of maximum cardinality in $S_{v}$. We then have $H'_1 = H_{v} \cup (H-v)$, where we used the fact that $H'_1-v= H-v$. Hence by Theorem~\ref{th decomp}, the pseudo-Boolean polytope $\PBP(H'_1)$ is decomposable into pseudo-Boolean polytopes $\PBP(H_{v})$ and $\PBP(H-v)$, and a polynomial-size extended formulation of $\PBP(H_{v})$ exists by \cref{pointedHull}.

Now assume $|N_1|\geq 2$; let $S_1 \subseteq S$ denote the set of signed edges whose underlying edges contain some $v \in N_1$. 
Consider $s = (e_s, \eta_s) \in S_1$, and define $\bar e := e_s \cup N_1$. Inflate $s$ to $\bar e$; repeat a similar inflation operation for every $s \in S_1$ to obtain a new signed hypergraph $\bar H=(V,\bar S)$. It then follows that $\bar H$ satisfies two key properties: 
\begin{itemize}
\item [(i)] $|\bar S| \leq 2^k |S|$, since by assumption ${\rm nsg}_{\N}(G) \leq k$, implying that $\gap(G, N_1) \leq k$,
\item [(ii)] all nodes in $N_1$ are nest points of $\bar H$. To see this, denote by $\bar G$ the underlying hypergraph of $\bar H$ and denote by
$\bar F$ the set of all edges in $\bar G$ containing some node in $N_1$. By definition of a nest-set, $\{f \setminus N_1 : f \in \bar F\}$ is totally ordered. Moreover, by the inflation operation defined above, we have $f \cap N_1 = N_1$ for all $f \in \bar F$. It then follows that $\bar F$ is totally ordered as well; hence, all nodes in $N_1$ are nest points of $\bar G$. 
\end{itemize}
By~\cref{lem inflation} and property~(i) above, by obtaining a polynomial-size extended formulation for $\PBP(\bar H)$, we obtain a polynomial-size extended formulation for $\PBP(H)$ as well.
Now consider a node $v_1 \in N_1$.
Since $v_1$ is a nest point of $\bar G$, we can use the technique described above to decompose $\PBP(\bar H)$ into $\PBP(\bar H_{v_1})$ and $\PBP(\bar H-v_1)$, where as before $\bar H_{v_1}$ is a signed hypergraph pointed at $v_1$. Moreover, denoting by $\bar S_{v_1}$ the number of signed edges in $\bar H_{v_1}$ whose underlying edge contains $v_1$, we have $\bar S_{v_1} \leq 2^k |S|$.
Next consider the signed hypergraph $\bar H-v_1$; it follows that any node $v \in N_1 \setminus \{v_1\}$ is a nest point of $\bar H-v_1$.
Hence we can apply our decomposition technique recursively to all nest points in $N_1$ to decompose $\PBP(\bar H)$ into $\PBP(\bar H-N_1)$ and $\PBP(\bar H_{v})$ for all $v \in N_1$, where $\PBP(\bar H_{v})$ is a pointed signed hypergraph at $v$ with at most $2^{k+1}|S|$ edges for all $v \in N_1$. Since the nodes in $N_1$ are contained in the same signed edges of $\bar H$ obtained by inflating any edge in $H$ containing some node in $N_1$ to their union, it follows that the number of signed edges of $\bar H-N_1$ is upper bounded by $|S|$. Next, we apply the above inflation and decomposition technique to the signed hypergraph $\bar H-N_1$ together with the nest-set $N_2$. Repeating this argument $t$ times for all $N_i \in \N$, we conclude that an extended formulation for $\PBP(H)$ is obtained by juxtaposing extended formulations of pointed signed hypergraphs $\PBP(\bar H_{v_i})$ for all $v_i \in V$, together with $t$ equalities of the form~\eqref{eq inflation sum}.

Now consider the pointed signed hypergraph $\bar H_{v_i}$ for some $v_i \in V$. Denote by $V_i$ the node set of $\bar H_{v_i}$, denote by $\bar S_{v_i}$ the set of signed edges whose underlying edge contains $v_i$, and let $\bar P_{v_i}$ denote the remaining signed edges of $\bar H_{v_i}$.
By Theorem~\ref{pointedHull} the polytope $\PBP(\bar H_{v_i})$ has an extended formulation consisting of at most $2|V_i| (|S_{v_i}|+|P_{v_i}|+1) \leq 2 r(2^{k}|S|+2^{k}|S|+1)$ variables and at most $4 (|S_{v_i}|+|P_{v_i}|)(|V_i|-2)+4 |V_i| \leq (r-2)2^{k+3}|S|+4|V|$ inequalities, where the inequalities follow since $|V_i| \leq r$ and $|\bar P_{v_i}| \leq |\bar S_{v_i}| \leq 2^k |S|$. Therefore, $\PBP(H)$ has an extended formulation with $O(r 2^k|V||S|)$ variables and inequalities.
Moreover, by~\cref{pointedHull} and~\cref{lem inflation} all coefficients and right-hand side constants in the system defining this extended formulation are $0,\pm 1$.
\end{proof}

Notice that~\cref{extended} is a special case of~\cref{conj1} obtained by letting $k=1$. We should remark that the proof of~\cref{conj1} indicates that if a nest-set elimination ordering $\N$ of $G$ with ${\rm nsg}_{\N} (G) \in O(\log \poly(|V|, |E|))$ is available, then one can construct an extended formulation for $\PBP(H)$ in polynomial-time, implying the polynomial-time solvability of the corresponding pseudo-Boolean optimization problem.
%We remark that the proof technique in~\cref{conj1} can be further refined to obtain polynomial-size extended formulations for the pseudo-Boolean polytope of a larger family of signed hypergraphs. In this proof, at each iteration, the inflation operation is utilized to convert a subset of nodes (\ie a nest-set) into a set of nest points. However, for the subsequent steps of the proof to work, it suffices instead to utilize inflation to convert a subset of nodes into a \emph{sequence of nest points}. Designing a proper inflation operation and characterizing the family of hypergraphs to which such a technique is applicable is a topic of future research.  
Using inequality~\eqref{relation}, we obtain the following result regarding the pseudo-Boolean polytope of signed hypergraphs with bounded nest-set width:

\begin{corollary}\label{corconj1}
Let $H=(V,S)$ be a signed hypergraph of rank $r$ whose underlying hypergraph $G=(V,E)$ satisfies ${\rm nsw}(G) \leq k$. Then the pseudo-Boolean polytope $\PBP(H)$ has an extended formulation with $O(r 2^k|V||S|)$ variables and inequalities.
In particular, if $k \in O(\log \poly(|V|, |E|))$, then $\PBP(H)$ has a polynomial-size extended formulation. Moreover, all coefficients and right-hand side constants in the system defining $\PBP(H)$ are $0,\pm 1$.
\end{corollary}

Recall that by~\cref{checknsw}, if $k \in O(\sqrt{\log \poly(|V|, |E|)})$, constructing a nest-set elimination ordering $\N$ with ${\rm nsw}_{\N}(G) \leq k$, if one exists, can be performed in polynomial-time. 
This in turn implies that, for signed hypergraphs with ${\rm nsw}(G)\in O(\sqrt{\log \poly(|V|, |E|)})$, the extended formulation in ~\cref{corconj1} can be constructed in polynomial time and, as a result, \cref{prob PBO} can be solved in polynomial-time.  
However, for the regime $k \in \omega(\sqrt{\log \poly(|V|, |E|)})$ and $k \in O(\log \poly(|V|, |E|))$, while by~\cref{corconj1} the extension complexity of the pseudo-Boolean polytope is polynomially bounded, and \cref{prob PBO} can be solved in polynomial-time if a corresponding nest-set elimination ordering is given, the complexity of constructing a nest-set elimination ordering $\N$ such that ${\rm nsw}_{\N}(G) \leq k$ remains an open question.

Let us examine the connections between~\cref{{th: alphaPoly}},~\cref{th PBO hard}, and~\cref{corconj1}. From Lemma~9 of~\cite{lanz23} it follows that for a hypergraph $G$
of rank $r$ we have 
$$\tw(G) \leq {\rm nsw}(G) (r-1).$$
Hence, if the rank $r$ is fixed, then~\cref{corconj1} is implied by~\cref{{th: alphaPoly}} as in this case, bounded nest-set width implies bounded treewidth. However, if $r$ is not fixed, then one can consider hypergraphs with bounded nest-set width and unbounded treewidth in which case~\cref{{th: alphaPoly}} is not applicable, while~\cref{corconj1} gives polynomial-size extended formulations for the pseudo-Boolean polytope.     

\medskip

We conclude this paper by remarking that in~\cite{dPKha23}, the authors present polynomial-size extended formulations for the pseudo-Boolean polytope of certain signed hypergraphs with ``log-poly gaps.'' To this end, they employ a two-step approach: In the first step, the inflation operation is invoked to ``remove" $\beta$-cycles from the underlying hypergraph of the signed hypergraph. In the second step, \cref{extended} is used to obtain a polynomial-size extended formulation for the pseudo-Boolean polytope of the newly constructed signed hypergraph whose underlying hypergraph is $\beta$-acyclic. The log-poly gap assumption ensures that the inflation operation adds at most polynomially many new edges to the signed hypergraph (see propositions~3 and~4 in~\cite{dPKha23}). This proof technique cannot be applied to prove~\cref{conj1}, because the structure of hypergraphs with bounded nest-set gaps is too complicated and the $\beta$-cycles cannot be removed in ``one shot.'' To prove this theorem, we need to apply the inflation operation in a \emph{recursive} manner. That is, after each decomposition step, a new nest-set is identified and an inflation operation is applied accordingly to obtain a set of nest points, which translates into the removal of a subset of $\beta$-cycles. 
However, we should remark that the first step of the above two-step approach can also be applied prior to the application of~\cref{conj1}, hence extending its scope. For example, if the signed hypergraph contains a long $\beta$-cycle with a large nest-set gap but satisfies the assumptions of proposition~3 in~\cite{dPKha23}, one can first use the inflation operation to effectively remove the $\beta$-cycle. If the resulting hypergraph has a bounded nest-set gap, then~\cref{conj1} can be used to obtain a polynomial-size extended formulation for the pseudo-Boolean polytope.  Obtaining a meta-theorem that fully characterizes the class of signed hypergraphs for which by utilizing the inflation operation, one can obtain a polynomial-size extended formulation for the pseudo-Boolean polytope is an interesting topic of future research. 

\bigskip
\noindent
\textbf{Acknowledgments:}
The authors would like to thank Matthias Lanzinger for fruitful discussions about nest-sets.

\bigskip
\noindent
\textbf{Funding:}
A. Del Pia is partially funded by AFOSR grant FA9550-23-1-0433. 
A. Khajavirad is in part supported by AFOSR grant FA9550-23-1-0123.
Any opinions, findings, and conclusions or recommendations expressed in this material are those of the authors and do not necessarily reflect the views of the Air Force Office of Scientific Research.

\ifthenelse {\boolean{SIOPT}}
{
% For SIOPT begin
\bibliographystyle{siamplain}
% For SIOPT end
}
{
% For OO begin
\bibliographystyle{plain}
% For OO end
}

%\begin{footnotesize}
%\end{footnotesize}

\end{document}